\documentclass{article}
\usepackage[final,nonatbib]{neurips_2018}

\usepackage{wrapfig}
\usepackage{enumitem,soul}
\usepackage{rotating}
\usepackage{chngcntr}
\usepackage{apptools}
\usepackage{listings}
\usepackage{graphicx}
\usepackage{graphics}
\usepackage{amssymb}
\usepackage{amsmath}
\usepackage{color,cite}
\usepackage{soul}
\usepackage{xspace}

\usepackage{algorithm,algorithmicx}
\usepackage{algcompatible}
\usepackage{bbm}
\usepackage{comment}
\usepackage{caption}
\usepackage{subcaption}
\usepackage{psfrag}
\usepackage{tikz}
\usetikzlibrary{math}
\usetikzlibrary{shapes,arrows}
\usetikzlibrary{arrows.meta}
\tikzset{>={Latex[width=2.5mm,length=2.5mm]}}
\usetikzlibrary{positioning}
\tikzstyle{block}=[draw opacity=0.7,line width=1.4cm]
\usepackage{rotating}
\usepackage{graphicx}[demo]
\usepackage{sidecap}
\usepackage{diagbox}

\newcommand{\margin}[1]{\marginpar{\color{red}\tiny\ttfamily#1}}

\newcommand{\needATT}[1]{{\color{purple}#1}}

\newcommand{\real}{{\mathbb{R}}} 
\newcommand{\realnonnegative}{{\mathbb{R}}_{\ge 0}}

\newcommand{\vect}[1]{\boldsymbol{\mathbf{#1}}}

\newcommand{\SUM}[2]{\sum_{#1}^{#2}}
\newcommand{\boxend}{\hfill \ensuremath{\Box}}
\newcommand{\Lnorm}{\left\|} \newcommand{\Rnorm}{\right\|}

\newtheorem{thm}{Theorem}[section]
\newtheorem{prop}{Proposition}[section]
\newtheorem{rem}{Remark}[section]

\newtheorem{lem}{Lemma}[section]

\newenvironment{proof}[1][Proof]{\begin{trivlist}
\item[\hskip \labelsep {\bfseries #1}]}{\boxend\end{trivlist}}

\makeatletter
\renewcommand*{\@opargbegintheorem}[3]{\trivlist
      \item[\hskip \labelsep{\bfseries #1\ #2}] \textbf{(#3)}\ \itshape}
\makeatother

\newcommand{\oprocendsymbol}{\hbox{$\bullet$}}
\newcommand{\oprocend}{\relax\ifmmode\else\unskip\hfill\fi\oprocendsymbol}



\usepackage{setspace}
\title{\LARGE \bf Distributed Strategy Selection: A Submodular Set Function Maximization Approach} 

\author{%
  Navid Rezazadeh\\
  Mechanical and Aerospace Eng. Dept.\\
  University of California Irvine\\
  Irvine, CA \\
  \texttt{nrezazad@uci.edu}
  \And
  Solmaz S. Kia\\
  Mechanical and Aerospace Eng. Dept.\\
  University of California Irvine\\
  Irvine, CA \\
  \texttt{solmaz@uci.edu} \\
}

\begin{document}
\maketitle
\thispagestyle{plain}
\pagestyle{plain}

\begin{abstract}                     
Constrained submodular set function maximization problems often appear in multi-agent decision-making problems with a discrete feasible set. A prominent example is the problem of multi-agent mobile sensor placement over a discrete domain. Submodular set function optimization problems, however, are known to be NP-hard. This paper considers a class of submodular optimization problems that consist of maximization of a monotone and submodular set function subject to a uniform matroid constraint over a group of networked agents that communicate over a connected undirected graph. We work in the value oracle model where the only access of the agents to the utility function is through a black box that returns the utility function value. We propose a distributed suboptimal polynomial-time algorithm that enables each agent to obtain its respective strategy via local interactions with its neighboring agents. Our solution is a fully distributed gradient-based algorithm using the submodular set functions' multilinear extension followed by a distributed stochastic Pipage rounding procedure. This algorithm results in a strategy set that when the team utility function is evaluated at worst case, the utility function value is in $\frac{1}{c}(1-{\textup{e}}^{-c}-O(1/T))$ of the optimal solution with $c$ to be the curvature of the submodular function. An example demonstrates our results.
\end{abstract}

\section{Introduction}
This paper studies a multi-agent optimal planning in discrete space, which is often referred to as optimal \emph{strategy scheduling}. We consider a group of $\mathcal{A}\!=\!\{1,...,N\}$ agents with communication and computation capabilities, interacting over a connected undirected graph $\mathcal{G}(\mathcal{A},\mathcal{E})$. Each agent $i\in\mathcal{A}$ has a distinct discrete strategy set $\mathcal{P}_i$, known only to agent $i$, and wants to choose at most $\kappa_i\in\mathbb{Z}_{>0}$ strategies from $\mathcal{P}_i$ such that a monotone increasing and submodular utility function $f:2^\mathcal{P}\to\real_{\geq0}$, $\mathcal{P} \!=\!\bigcup\nolimits_{i \in \mathcal{A}}  \mathcal{P}_i$, evaluated at all the agents' strategy selection is maximized\footnote{We use standard notation, but for clarity we provide a brief description of the notation and the definitions in Section~\ref{sec::Notation}.}. 
In other words, the agents aim to solve in a distributed manner the discrete domain optimization~problem
\begin{subequations}\label{eq::mainProblem}
\begin{align}
    &\underset{\mathcal{R}\in \mathcal{I}}{\textup{max}}f(\mathcal{R}) \\
    \mathcal{I} = \big\{ \mathcal{R} \subset \mathcal{P}\,&\big|\,\, |\mathcal{R} \cap \mathcal{P}_i|\leq \kappa_i,~ \forall i\in\mathcal{A}\big\}.\label{eq::mainProblem_b}
\end{align}
\end{subequations}
 The agents' access to the utility function is through a black box that returns $f(\mathcal{R})$ for any given set $\mathcal{R}\in\mathcal{P}$ (value oracle model).
The constraint set~\eqref{eq::mainProblem_b} is a partition matroid, which restricts the number of strategy choices of each agent $i\in\mathcal{A}$ to $\kappa_i$. In a distributed solution, each agent $i\in\mathcal{A}$ should obtain its respective component $\mathcal{R}^\star_i\in\mathcal{P}_i$ of $\mathcal{R}^\star=\cup_{j=1}^N \mathcal{R}^\star_j$, the optimal solution of~\eqref{eq::mainProblem}, by interacting only with the agents that are in its communication range. Optimization problem~\eqref{eq::mainProblem} is known to be NP hard~\cite{GLN-LAW-MLF:78}. Our goal is to design a polynomial-time distributed solution for~\eqref{eq::mainProblem} with formal guarantees on the optimality bound. Our problem is motivated by a wide range of multi-agent exploration, and sensor placement problems for monitoring and coverage, which can be formulated in the form of submodular maximization problem~\eqref{eq::mainProblem}~\cite{NR-SSK:21,AH-MG-HV-UT:20,VT-AJ-GJP:16,XD-MG-RP-FB:20,SW-CGC:20,SA-EF-LSL:14}. An example application scenario is shown in Fig.~\ref{fig:city_partition}.

When the utility function of~\eqref{eq::mainProblem} is monotone increasing and submodular set function, the \emph{sequential greedy algorithm}~\cite{MLF-GLN-LAW:78} delivers a polynomial-time suboptimal solution with guaranteed optimality bound.
The sequential greedy algorithm is a central solution. Several attempts have been undertaken to adapt the sequential greedy algorithm for large-scale submodular maximization problems by reducing the size of the problem through approximations~\cite{KW-RI-JB:14} or using several processing units to achieve a faster sequential greedy algorithm, however, with some sacrifices on the optimality bound~\cite{BM-AK-RS-AK:13,BM-MZ-AK:16,RK-BM-SV-AV:15,PSR-OS-MF:20}. Decentralized implementation of the sequential greedy algorithm through sequential message passing or via sequential message sharing through a cloud is also discussed in~\cite{NR-SSK:21}.

The optimality gap achieved by sequential greedy algorithm for problem~\eqref{eq::mainProblem} is $\frac{1}{1+c}\,f(\mathcal{R}^\star)$~\cite{MC-GC:84}, where $c$ is the \emph{curvature} of the utility function. More recently, another suboptimal solution for problem~\eqref{eq::mainProblem} with an  improved optimality gap 
is proposed in the literature using the \emph{multilinear} continuous relaxation of a submodular set function~\cite{JV:08,AAB-BM-JB-AK:17,AM-HH-AK:20,OS-MF:20}. 
The relaxation transforms the discrete problem~\eqref{eq::mainProblem} into a continuous optimization problem with linear constraints. Then, a continuous gradient-based optimization algorithm is used to solve the continuous optimization problem. A suboptimal solution for~\eqref{eq::mainProblem} with the improved  optimality gap of $\frac{1}{c}(1-{\textup{e}}^{-c})\, f(\mathcal{R}^\star)$ then is obtained by properly \emph{rounding} the continuous-domain solution~\cite{JV:10}. This approach however requires a central authority to solve the problem. In sensor placement problems, when the agents are self-organizing autonomous mobile agents with communication and computation capabilities, it is desired to solve the optimization problem~\ref{eq::mainProblem} in a distributed way, without involving a central authority. For the special case of zero curvature, and when each agent has only one choice to make, i.e.,  $\kappa_i = 1$ for all $i\in\mathcal{A}$, \cite{AR-AA-BS-JGP-HH:19} has proposed an average consensus-based distributed algorithm to solve the multi-linear extension relaxation of~\ref{eq::mainProblem} over connected graphs. The solution of~\cite{AR-AA-BS-JGP-HH:19} requires exact knowledge of the multi-linear extension function. However, the computational complexity of constructing the exact value of the multi-linear extension of a submodular function and its derivatives is exponential in the size of the strategy set. The result in~\cite{AR-AA-BS-JGP-HH:19} also requires a centralized rounding scheme.

The multi-linear extension function of a submodular function, as we review in Section~\ref{sec::prelim_submod}, is equivalently the expected value of the submodular function evaluated at random sets obtained by picking strategies from the strategy set independently with a probability. This stochastic interpretation has allowed approximating the multi-linear extension function and its derivatives emphatically with a reasonable computational cost via sampling from the strategy set. This approach has been used~\cite{JV:10} in its continuous domain solution of submodular optimization problems. Of course, as expected, this approach comes with a penalty on the optimality gap that is inversely proportional to the number of the samples. 

In this paper, motivated by the improved optimality gap of the multilinear continuous relaxation-based algorithms, we develop a distributed implementation of the algorithm of~\cite{JV:10} over a connected undirected graph to obtain a suboptimal solution for~\eqref{eq::mainProblem}. We propose a gradient-based algorithm, which uses a maximum consensus message passing scheme over the communication graph. We complete our solution by designing a fully distributed rounding procedure to compute the final suboptimal strategy of each agent. In this algorithm, to manage the computational cost of constructing the multilinear extension of the utility function, we use a sampling-based evaluation of the multilinear extension. Through rigorous analysis, we show that our proposed distributed algorithm in finite time $T$ achieves, with a known probability, a $\frac{1}{c}(1-{\textup{e}}^{-c})-O(1/T)$ optimality bound, where $1/T$ is the step size of the algorithm and the frequency at which agents communicate over the network. A numerical example demonstrates our results.

\section{Preliminaries}\label{sec::Notation}
This section introduces our notation and relevant definitions from graph theory and submodular set functions\footnote{Since we use standard notation, the reader familiar with the subject can continue to Section~\ref{sec::Main} and come back to this section if the need arises.}. 
\subsection{Notion}
We denote the vectors with bold small font.  The $p^{\textup{th}}$ element of vector $\vect{x}$ is denoted by $[\vect{x}]_p$. We denote the inner product of two vectors $\vect{x}$ and $\vect{y}$ with appropriate dimensions by $\vect{x}.\vect{y}$. We use $\vect{1}$ as a vector of ones, whose dimension is understood from the context. We denote sets with the capital curly font. 
Given a ground set $\mathcal{P}=\{1,\cdots,n\}$, we define the membership probability vector $\vect{x} \in[0,1]^n$ to obtain  $\mathcal{R}_{\vect{x}}\subset\mathcal{P}$ as a random set where $p \in\mathcal{P}$ is in  $\mathcal{R}_{\vect{x}}$ with the probability $[\vect{x}]_p$. For $\mathcal{R} \subset \mathcal{P}$, $\vect{1}_{\mathcal{R}} \in \{0,1\}^{n}$ is the vector whose $p^{\textup{th}}$ element is $1$ if $p \in \mathcal{R}$ and $0$ otherwise; we call $\vect{1}_{\mathcal{R}}$ the membership indicator vector of set $\mathcal{R}$. Given a finite countable set $\mathcal{R}\subset\real$ and integer $\kappa$, $1\leq \kappa<|\mathcal{R}|$, $\max(\mathcal{R},\kappa)$ returns the $\kappa$ largest elements of $\mathcal{R}$. $|x|$ is the absolute value of  $x \in \real$. By overloading the notation, we also use $|\mathcal{R}|$ as the cordiality of set $\mathcal{R}$. For a set function $f:2^\mathcal{P}\to \real$, we define $\Delta_f(p|\mathcal{R}) = f(\mathcal{R} \cup \{p\})-f(\mathcal{R}),\,\, \mathcal{R} \subset \mathcal{P}$. A set function is \emph{normalized} if $f(\emptyset)=0$. 
Given a set $\mathcal{F}\subset \mathcal{X}\times \real$ and an element $(p,\alpha)\in \mathcal{X}\times \real$ we define the addition operator $\oplus$ as $\mathcal{F}{\oplus}\{(p,\alpha)\}=\{(u,\gamma)\in \mathcal{X}\times \real\,|\,(u,\gamma) \in \mathcal{F}, u\not = p \} \cup \{(u,\gamma+\alpha)\in \mathcal{X}\times \real\,|\, (u,\gamma) \in \mathcal{F}, u = p \} \cup \{(p,\alpha)\in \mathcal{X}\times \real\,|\,(p,\gamma) \not \in \mathcal{F}, \gamma \in \mathbb{R} \}$. Given a collection of sets $\mathcal{F}_i\in \mathcal{X}\times \real$, $i\in\mathcal{A}$, we define the max-operation over these collection as $\underset{i\in \mathcal{A}}{\textup{MAX}} \,\,\mathcal{F}_i= \{(u,\gamma)\in \mathcal{X}\times \real| (u,\gamma) \in \bar{\mathcal{F}} \text{~s.t.~} \gamma =  \underset{(u,\alpha) \in \bar{\mathcal{F}}}{\textup{max}} \alpha\}$, where $\bar{\mathcal{F}}=\bigcup_{i\in \mathcal{A}} \mathcal{F}_i.$

We denote a graph by $\mathcal{G}(\mathcal{A},\mathcal{E})$ where $\mathcal{A}$ is the node set and $\mathcal{E} \subset \mathcal{A} \times \mathcal{A}$ is the edge set. 
$\mathcal{G}$ is undirected if and only is $(i,j) \in \mathcal{E}$ means that agents $i$ and $j$ can exchange information. An undirected graph is connected if there is a path from each node to every other node in the graph. We denote the set of the neighboring nodes of node $i$ by $\mathcal{N}_i=\{j\in\mathcal{A}|(i,j) \in \mathcal{E} \}$. We also use $d(\mathcal{G})$ to show the diameter of the graph.

\subsection{Submodular Functions}\label{sec::prelim_submod}
A set function $f:2^{\mathcal{P}} \to \realnonnegative$ is monotone increasing if $f(\mathcal{P}_1)\leq f(\mathcal{P}_2)$, and submodular if 
for any $p \in \mathcal{P} \setminus \mathcal{P}_2$, 
\begin{align}\label{eq::submodularity}
    \Delta_f(p|\mathcal{P}_1)\geq \Delta_f(p|\mathcal{P}_2),
\end{align}
hold for any sets $\mathcal{P}_1 \subset \mathcal{P}_2 \subset \mathcal{P}$. The total curvature of a submodular set function $f:2^{\mathcal{P}} \to \realnonnegative$, which shows the worst-case increase in the value of the function when member $p$ is added, is defined as 
\begin{align}\label{eq::totalCurvature}
    c = 1 - \underset{\mathcal{S},p \not \in \mathcal{S}}{\textup{min}}\frac{\Delta_f(p|{\mathcal{S}})}{\Delta_f(p|\emptyset)}
\end{align}
Note that $c\in[0,1]$; the curvature of $c=0$ means that the function is modular, i.e.,  $f(\{p_1,p_2\})= f(\{p_1\})+f(\{p_2\})$, $p_1,p_2\in\mathcal{P}$, while $c=1$ means that there is at least a member that adds no value to function $f$ in a special circumstance. Curvature $c$ represents a measure of the diminishing return of a set function. Whenever the total curvature is not known, it is rational to assume the worst case scenario and set $c=1$. 

In the rest of this paper, without loss of generality, we assume that the ground set $\mathcal{P}$ is equal to $\{1,\cdots,n\}$. 
For a submodular function $f:2^{\mathcal{P}} \to \realnonnegative$, its \emph{multilinear extension}  $F:[0,1]^{n} \to \realnonnegative$ in the continuous space is 
\begin{align}\label{eq::F_determin}
    F(\vect{x}) = \!\!\sum_{\mathcal{R} \subset \mathcal{P}}{}\! f(\mathcal{R}) \prod_{p \in \mathcal{R}}^{}\! [\vect{x}]_p \prod_{p \not\in \mathcal{R}}^{}\!(1-[\vect{x}]_p),~~ \vect{x} \in [0,1]^{n}.
\end{align}
$F$ in~\eqref{eq::F_determin} is indeed equivalent to 
\begin{align}\label{eq::F_stoc}
   F(\vect{x})=\mathbb{E}[f(\mathcal{R}_{\vect{x}})],  
\end{align} where $\mathcal{R}_{\vect{x}}\subset \mathcal{P}$ is a set where each element $p\in\mathcal{R}_{\vect{x}}$ appears independently with the probabilities $[\vect{x}]_p$. Then,  taking the derivatives of $F(\vect{x})$ yields 
\begin{align}\label{eq::firstDer}
\frac{\partial F}{\partial [\vect{x}]_p} (\vect{x})= \mathbb{E}[f(\mathcal{R}_{\vect{x}} \cup \{p\})-f(\mathcal{R}_{\vect{x}} \setminus \{p\})],
\end{align}
and
\begin{align}\label{eq::secondDer}
& \frac{\partial^2 F}{\partial [\vect{x}]_p \partial [\vect{x}]_q}(\vect{x}) = \mathbb{E}[f(\mathcal{R}_{\vect{x}} \cup \{p,q\})-f(\mathcal{R}_{\vect{x}} \cup \{q\} \setminus \{p\})\nonumber\\
&\quad\quad \quad \quad-f(\mathcal{R}_{\vect{x}} \cup \{p\} \setminus \{q\})+f(\mathcal{R}_{\vect{x}} \setminus \{p,q\})].
\end{align}

To construct $F(\vect{x})$ and its derivatives we should evaluate  $f(\mathcal{R})$ for all $\mathcal{R} \in 2^ \mathcal{P}$. Therefore, when the size of $\mathcal{P}$ grows, constructing $F(\vect{x})$ and its derivatives computationally become intractable. The stochastic interpretation~\eqref{eq::F_stoc} of the multilinear extension and its derivatives offer a mechanism to estimate them with a reasonable computational cost via sampling.  Chernoff-Hoeffding's inequality can be used to quantify the quality of these estimates given the number of samples.
\begin{thm}[Chernoff-Hoeffding's inequality~\cite{WH:94}]\label{thm::sampling}
Consider a set of $K$ independent random variables $X_1,...,X_K$ where $a<X_i<b$. Let $S_K=\SUM{i=1}{K} X_i$, then
\begin{align*}
    \mathbb{P}[|S_K - \mathbb{E}[S_K]|>\delta] < 2 \, \textup{e}^{-\frac{2\delta ^2}{K(b-a)^2}},
\end{align*}
for any $\delta\in\real_{\geq 0}$.
\end{thm}
Given a ground set $\mathcal{P}$,
Partition matroid $\mathcal{M}$ is defined as $\mathcal{M} = \{ \mathcal{R} \subset \mathcal{P} \big| \,\, |\mathcal{R} \cap \mathcal{P}_i|\leq k_i, \forall i\in\mathcal{A}\}$, $1\leq \kappa_i\leq |\mathcal{P}_i|$ and $\mathcal{A}=\{1,\dots,N\}$, with  $\bigcup\nolimits_{i \in \mathcal{A}}  \mathcal{P}_i =  \mathcal{P}$ and $ \mathcal{P}_i \cap  \mathcal{P}_j = \emptyset,\,\, i\neq j$. 
The \emph{matroid polytop} for partition matroid is a convex hull defined as (with abuse of notation)
\begin{align}\label{eq::convexhull}\mathcal{M}=\{ \vect{x} \in [0,1]^{n}\,\big|\,\sum\nolimits_{p \in \mathcal{P}_i}{}[\vect{x}_i]_{p}\leq \kappa_i, \forall i\in\mathcal{A}\}.
\end{align}
where $[\vect{x}]_p, \,\, p \in \mathcal{P}_i$ associated with membership probability of strategies in $\mathcal{P}_i$, $i\in\mathcal{A}=\{1,\cdots,N\}$, we have $\SUM{p \in \mathcal{P}_i}{} [\vect{x}]_p \leq \kappa_i$.

The following result is established for the optimal solution of the optimization problem~\eqref{eq::mainProblem}
\begin{lem}[\cite{JV:10}]\label{lm::optimalInEq}
{\it
Consider the set value optimization problem~\eqref{eq::mainProblem}. Suppose $f:\mathcal{P} \to \realnonnegative$ is an increasing and submodular set function with curvature $c$ whose multilinear extension is  $F:[0,1]^{n} \to \realnonnegative$. Let $\mathcal{R}^\star$ be the optimizer of~\eqref{eq::mainProblem}. Then, 
\begin{align*}
    \vect{1}_{\mathcal{R}^\star}.\nabla F(\vect{x}) \geq f(\mathcal{R}^\star) - c\,F(\vect{x}),\quad \forall \vect{x} \in \mathcal{M}.
\end{align*}
}\end{lem}

\section{Decentralized continuous-domain strategy selection}\label{sec::Main}
The utility function $f$  assigns values to all the subsets of $\mathcal{P} \!=\!\bigcup\nolimits_{i \in \mathcal{A}}  \mathcal{P}_i=\{1,\cdots,n\}$. Thus, equivalently, we can regard the set value utility function as a function on the Boolean hypercube $ \{0,1\}^n$, i.e., $f : \{0,1\}^n \to\real$. Multilinear-extension function $F(\vect{x})$, given in~\eqref{eq::F_determin}, expands the function evaluation of the utility function to over the space between the vertices of this Boolean hypercube. On the other hand, the matroid polytope $\mathcal{M}$ is the convex haul of the vertices of the hypercube $\{0,1\}^n$ that satisfy the partition matroid constraint~\eqref{eq::mainProblem_b}.  
Additionally, note that according to~\eqref{eq::F_determin}, $F(\vect{x})$ for any $\vect{x} \in \mathcal{M}$ is a weighted average of values of $F$ at the vertices of the matriod polytope $\mathcal{M}$. Then, equivalently, 
$F(\vect{x})$ at any $\vect{x} \in \mathcal{M}$ is a  normalized-weighted average of $f$ on the strategies satisfying constraint~\eqref{eq::mainProblem_b}. 
As such, 
\begin{align*}
    f(\mathcal{R}^\star) \geq F(\vect{x}), \,\,\, \vect{x} \in \mathcal{M},
\end{align*}
which is equivalent to $f(\mathcal{R}^\star) =\underset{\vect{x} \in \mathcal{M}}{\textup{max}} \,\,F(\vect{x})$, where $\mathcal{R}^\star$ is the optimizer of problem~\eqref{eq::mainProblem}~\cite{JV:08}. Therefore, solving the continuous domain optimization problem \begin{align}\label{eq::continu_opt}\underset{\vect{x} \in \mathcal{M}}{\textup{max}} \,\,F(\vect{x}),
\end{align}
by using a gradient-based method combined with a proper \emph{rounding} procedure presents itself as a gateway to solving the optimization problem~\eqref{eq::mainProblem}. For instance, the solution proposed in~\cite{JV:10} to solve~\eqref{eq::continu_opt} is the constrained gradient ascent solver
\begin{align}\label{eq::VNabF}
\frac{\text{d}\vect{x}}{\text{d}t}=\vect{v}(\vect{x})~~\text{where}~~ \vect{v}(\vect{x})=\underset{\vect{w}\in \mathcal{M}}{\arg\max}(\vect{w}.\nabla F(\vect{x})),\end{align}
initialized at $\vect{x}(0)\!=\!\vect{0}$. Since $\mathcal{M}$ is convex and $\vect{x}(0)\!=\!\vect{0}\!\in\!\mathcal{M}$, the trajectory $t\!\mapsto\!\vect{x}$ of~\eqref{eq::VNabF} belongs to  $\mathcal{M}$ for $t\in[0,1]$. Due to Lemma~\ref{lm::optimalInEq}, ascent direction  in~\eqref{eq::VNabF} satisfies
\begin{align*}
    \frac{\text{d}F}{\text{d}t}=\vect{v}(\vect{x}).\nabla F(\vect{x})\geq f(\mathcal{R}^\star)-c\,F(\vect{x}).
\end{align*}
Using the Comparison Lemma~\cite{HKK-JWG:02} then~\eqref{eq::VNabF} concludes that $F(\vect{x}(1))\geq \frac{1}{c}(1-\text{e}^{-c})\,f(\mathcal{R}^\star)$. 
To complete its suboptimal solution,~\cite{JV:10} uses \emph{Pipage} rounding procedure~\cite{AAA-MIS:04} on $\vect{x}(1)$ to obtain the suboptimal solution of~\eqref{eq::mainProblem} with optimality gap of $\frac{1}{c}(1-\text{e}^{-c})\,f(\mathcal{R}^\star)$. 

Inspired by the central solver~\eqref{eq::VNabF}, in what follows, we present a practical finite time distributed solution for problem~\eqref{eq::mainProblem}. Our solution includes a proper distributed sampling method to obtain a polynomial-time numerical approximation of $\nabla\vect{F}(\vect{x})$ and also a distributed rounding~procedure.

\subsection{Distributed Discrete Gradient Ascent Solution}

To design a distributed iterative solution for~\eqref{eq::mainProblem}, let every agent $i\in\mathcal{A}$ maintain and evolve a local copy of the membership probability vector as $\vect{x}_i(t) \in\real^n$. Since $\mathcal{P}=\{1, \cdots, n\}$ is sorted agent-wise, we denote $\vect{x}_i(t)=[\hat{\vect{x}}_{i1}^\top(t),\cdots,\vect{x}_{ii}^\top(t),\cdots,\hat{\vect{x}}_{iN}^\top(t)]^\top\in\real^n$ where $\vect{x}_{ii}(t)\in \realnonnegative^{|\mathcal{P}_i|}$ is the membership probability vector of agent $i$'s own strategy with entries of $[\vect{x}_i(t)]_p, \,\, p \in \mathcal{P}_i$ at iteration $t\in\{0,1,\cdots,T\}$, $T\in\mathbb{Z}_{>0}$, while $\hat{\vect{x}}_{ij}(t) \in \realnonnegative^{|\mathcal{P}_j|}$ is the local  estimate of the membership probability vector of agent $j$ by agent $i$ with entries of $[\vect{x}_i(t)]_p \,\, p \in \mathcal{P}_j, \,\, j \in \mathcal{A} \setminus i$. Every agent $i\in\mathcal{A}$ initializes at $\vect{x}_i(0)=\vect{0}$.

At time step $t$, each agent $i \in \mathcal{A}$ generates $K_i$ independent samples $\{\mathcal{R}^k_{i}(t)\}_{k=1}^{K_i}$ of random set $\mathcal{R}_{\vect{x}_i(t)}$.

drawn according to membership probability vector $\vect{x}_i(t)$ from $\mathcal{P}$ and empirically computes gradient vector $\nabla F(\vect{x}_i(t)) \in \realnonnegative^{n}$, which according to~\eqref{eq::firstDer} is defined element-wise as 
\begin{align}\label{eq::emperic_nabla_F}
\left[\widetilde{\nabla F}(\vect{x}_i(t))\right]_p \!\!=\!\!\left(\!\sum_{\,k=1}^{K_i}\! f(\mathcal{R}^k_{i}(t) \!\cup \!\{p\})\!-\!f(\mathcal{R}^k_{i}(t)\!\setminus\! \{p\})\!\!\right)\!\!\Big/\!K_i,
\end{align}
$p\in\mathcal{P}=\{1,\cdots,n\}$. Hereafter, $\widetilde{\nabla F}(\vect{x}_i(t))$ is the empirical estimate of $\nabla F(\vect{x}_i(t))$. The following lemma, whose proof is given in the Appendix B and relies on the Chernoff-Hoeffding's inequality, can quantify the quality of  this estimate. 
\begin{lem}\label{lem::gradient_estimate}{\it
Consider the set value optimization problem~\eqref{eq::mainProblem}. Suppose $f:\mathcal{P} \to \realnonnegative$ is an increasing and submodular set function and consider its multilinear extension $F:[0,1]^n \to \realnonnegative$. Let $\widetilde{\nabla F}(\vect{x})$ be the  estimate of $\nabla F(\vect{x})$ that is calculated by taking $K\in\mathbb{Z}_{>0}$ samples of set $\mathcal{R}_{\vect{x}}$ according to membership probability vector $\vect{x}$.  Then,  
\begin{align}\label{eq::estimationACC}
  \!\!  \left|\left[\widetilde{\nabla F}(\vect{x}(t))\right]_p - \frac{{\partial F}}{\partial [\vect{x}]_p}(\vect{x})\right|\! \geq \frac{1}{2T}f(\mathcal{R}^\star), ~~ p\in \{1,\cdots,n\},
\end{align}
with the probability of $2\textup{e}^{-\frac{1}{8T^2}K}$, for any $T\in\mathbb{Z}_{>0}$.
}
\end{lem}

Let each agent $i \in \mathcal{A} $ \emph{propagate} its local probability membership function with step size $\frac{1}{T}$ according to 
\begin{align}\label{eq::updateRule1}
    \vect{x}_i^-(t+1) = \vect{x}_i(t)+\frac{1}{T}\widetilde{\vect{v}}_i(t),
\end{align}
where 
\begin{align}\label{eq::updateVect}
\widetilde{\vect{v}}_i(t) = \underset{\vect{w} \in \mathcal{M}_i}{\textup{argmax}}\, \vect{w}.\widetilde{\nabla F}(\vect{x}_i(t)) 
\end{align}
 with
\begin{align}\label{eq::updateVectSpace}\mathcal{M}_{i}\!=\!\Big\{ \vect{w}\!\in\! [0,1]^n \Big|\,\vect{1}^\top.\vect{w}\leq \kappa_i \,\,, [\vect{w}]_p=0,\,\, \forall p \in \mathcal{P}\backslash\mathcal{P}_{i} \Big\}.
\end{align}

Because the propagation is only based on the local information of agent $i$, next, we \emph{update} the propagated $\vect{x}_i^-(t+1)$ of each agent $i\in\mathcal{A}$ by element-wise maximum seeking among its neighbors, i.e.,
\begin{align}\label{eq::updateRule2}
    \vect{x}_i(t+1) = \underset{j \in \mathcal{N}_i \cup \{i\}}{\textup{max}} \vect{x}_j^-(t+1).
\end{align}
 Lemma~\ref{lem::maxVect} below, whose proof is given in the Appendix A, shows that, as expected, \begin{align*}\vect{x}_{ii}(t)=\vect{x}_{ii}^-(t),\quad i\in\mathcal{A},\end{align*}
i.e., the corrected component of $\vect{x}_{i}$ corresponding to agent $i$ itself is the propagated value maintained at agent $i$, and not the estimated value of any of its neighbors. 

\begin{lem}\label{lem::maxVect}{\it
Let the agents propagate and update their local membership probability vector according to~\eqref{eq::updateRule1} and~\eqref{eq::updateRule2}. Let \begin{align}\label{eq::x_bar}\bar{\vect{x}}(t) = \underset{i \in \mathcal{A}}{\textup{max}} \,\, \vect{x}_i(t).
\end{align}
Then, at any $t\in\{0,1,\cdots,T\}$, we have 
\begin{align}\label{eq::x_bar_value}
\bar{\vect{x}}(t)=[\vect{x}_{11}^\top(t),\cdots,\vect{x}_{NN}^\top(t)]^\top.
\end{align} 
}
\end{lem}
After the update~\eqref{eq::updateRule2}, at each time step $t$ each agent has a local membership probability vector of the strategies $\vect{x}_i(t)$, which is not necessarily the same for all $i\in\mathcal{A}$. The following result, whose proof is given in Appendix B, establishes the difference between the agents' local copies of the membership probability~vector.

\begin{prop} \label{thm::convergance}{\it
Let the agents propagate and update their local membership probability vector according to~\eqref{eq::updateRule1} and~\eqref{eq::updateRule2}.  Then, the membership probability $\vect{x}_i(t)$, $t\in\{0,\cdots,T\}$, for each agent $i \in \mathcal{A}$ satisfies
\begin{subequations}\label{eq::Pro51}
\begin{align}
  & 0\leq \frac{1}{\kappa}\vect{1}.(\bar{\vect{x}}(t)-\vect{x}_i(t)) \leq \frac{1}{T} d(\mathcal{G}),\label{eq::Pro51_a}\\
    &\bar{\vect{x}}(t+1) - \bar{\vect{x}}(t) = \frac{1}{T} \sum\nolimits_{i \in \mathcal{A}}\! \widetilde{\vect{v}}_i(t),\label{eq::Pro51_b}\\
    &\frac{1}{\kappa_i}\vect{1}.(\bar{\vect{x}}(t+1) - \bar{\vect{x}}(t)) =\frac{1}{T},
    \label{eq::Pro51_c}
\end{align}
\end{subequations}
where $\kappa = \SUM{i \in \mathcal{A}}{} \kappa_i$.}
\end{prop}
Next lemma, whose proof is given in Appendix A, establishes that both  $\vect{x}_i(t)$ and $\bar{\vect{x}}(t)$ belong to $\mathcal{M}$ for any $t\in\{0,\cdots,T\}$. 
\begin{lem}\label{lem::x_update_property}
Let the agents propagate and update their local membership probability vector according to~\eqref{eq::updateRule1} and~\eqref{eq::updateRule2}. Then, (a) $\vect{x}_i(t), \,\,\, \bar{\vect{x}}(t)\in\mathcal{M}$ at any $t\in\{0,1,\cdots,T\}$; (b) $\vect{1}.\vect{x}_{ii}(T) =\kappa_i$,  $\vect{1}.\hat{\vect{x}}_{ij}(T) \leq \kappa_j$, $j \in \mathcal{A}\setminus \{i\}$, and $\vect{x}_{i}(T)\in[0,1]^n$.
\end{lem}

The following theorem, whose proof is given in Appendix A, quantifies the optimality of $\bar{\vect{x}}(T) \in \mathcal{M}$ evaluated by the multilinear-extension function $F$. 

\begin{thm}[Optimality gap]\label{thm::main}
Let the agents propagate and update their local membership probability vector according to~\eqref{eq::updateRule1} and~\eqref{eq::updateRule2}. 
Let $\kappa = \SUM{i \in \mathcal{A}}{} \kappa_i$, and $\mathcal{R}^\star$ be the optimizer of problem~\eqref{eq::mainProblem}.~Then, the admissible strategy set $\bar{\mathcal{P}}=\bigcup_{i \in \mathcal{A}} \bar{\mathcal{P}}_i$, 
\begin{align}\label{eq::bound_F_xbar_T}
\frac{1}{c}(1\!-\!{\textup{e}}^{-c})(1\!-\!(2 \,c\,\kappa\, d(\mathcal{G})\!+\!\frac{c\,\kappa}{2}+1)\frac{\kappa}{T})f(\mathcal{R}^\star) \!\leq F(\bar{\vect{x}}(T))
\end{align}
 with the probability of $\left(\prod_{i\in\mathcal{A}} (1-2\textup{e}^{-\frac{1}{8T^2}K_i})^{|\mathcal{P}_i|}\right)^T$.
\end{thm}
Notice that since $$1-2T\,n\, \textup{e}^{-\frac{1}{8T^2}\underline{K}}\leq \left(\prod\nolimits_{i\in\mathcal{A}} (1-2\textup{e}^{-\frac{1}{8T^2}K_i})^{|\mathcal{P}_i|}\right)^T,$$ where $\underline{K}=\min\{K_1,\cdots,K_N\}$, 
the probability of the bound~\eqref{eq::bound_F_xbar_T} improves as $T$, and the number of the samples collected by the agents $\underline{K}$ increase.

The final output of a distributed solver for problem~\eqref{eq::mainProblem} must be a set $\bar{\mathcal{R}}$ satisfying~\eqref{eq::mainProblem_b}. Recall that strategies corresponding to the vertices of the matroid polytope $\mathcal{M}$ belong to admissible strategy set of~\eqref{eq::mainProblem}. However, $\bar{\vect{x}}(T)$ is a fractional point in $ \mathcal{M}$. Moreover, only part of $\bar{\vect{x}}(T)$ is available at each agent $i\in\mathcal{A}$. In what follows, we propose a distributed stochastic Pipage rounding to move $\bar{\vect{x}}(T)$ to a a vertex of $\mathcal{M}$ starting from $\vect{x}_i(T)$. 

\subsubsection{Distributed Pipage rounding procedure}
Let each agent $i\in\mathcal{A}$ initialize its local rounded membership vector $\vect{y}_i\in\real^{n}$ at $\vect{y}_i(0)=\vect{x}_i(T)$. Then, by virtue of Lemma~\ref{lem::x_update_property}, we have $\vect{y}_i(0)\in[0,1]^n$, $i\in\mathcal{A}$. Following a Pipage type rounding procedure, each agent 
 $i \in \mathcal{A}$ at each iteration picks two two local strategies $p,q \in \mathcal{P}_i$ for which  $[\vect{y}_i(\tau)]_p,[\vect{y}_i(\tau)]_q \in (0,1)$, and preforms 
\begin{align}\label{eq::dist_pipage}
 &   \!\!\!    \begin{cases}
     [\vect{y}_i(\tau+1)]_p = [\vect{y}_i(\tau)]_p \!-\! \delta_p(\tau), \\
     [\vect{y}_i(\tau+1)]_q = [\vect{y}_i(\tau)]_q + \delta_p(\tau),
     \end{cases} \!\!\!\textup{w.p.}~\frac{\delta_q(\tau)}{\delta_p(\tau)\!+\!\delta_q(\tau)},\nonumber\\
  &     \!\!\!    \begin{cases}
    [\vect{y}_i(\tau+1)]_q = [\vect{y}_i(\tau)]_q\!-\! \delta_q(\tau), \\
     [\vect{y}_i(\tau+1)]_p = [\vect{y}_i(\tau)]_p + \delta_q(\tau), 
        \end{cases}\!\!\! \textup{w.p.}~ \frac{\delta_p(\tau)}{\delta_p(\tau)\!+\!\delta_q(\tau)},\nonumber\\
    &     
~\,[\vect{y}_i(\tau+1)]_l=[\vect{y}_i(\tau)]_l,\quad l\in\{1,\cdots,n\}\backslash\{p,q\},
\end{align}
where $\delta_p(\tau)=\textup{min}([\vect{y}_i(\tau)]_p,1-[\vect{y}_i(\tau)]_q)$ and $\delta_q(\tau)=\textup{min}(1-[\vect{y}_i(\tau)]_p,[\vect{y}_i(\tau)]_q)$. Here, `w.p.' stands for `with probability of'. The following proposition gives the convergence result of distributed Pipage rounding~\eqref{eq::dist_pipage}. In what follows, we partition $\vect{y}_i(\tau)$ as $\vect{y}_i(\tau)=[\hat{\vect{y}}_{i1}^\top(\tau),\cdots,\vect{y}_{ii}^\top(\tau),\cdots,\hat{\vect{y}}_{iN}^\top(\tau)]^\top\in[0,1]^n$, $i\in\mathcal{A}$, where $\vect{y}_{ij}(\tau)\in\real^{|\mathcal{P}_j|}$ for $j\in\mathcal{A}$.

\begin{prop}\label{prop::rounding_prop}
Starting from $\vect{y}_i(0)=\vect{x}_i(T)$, let each agent $i\in\mathcal{A}$ implement the rounding procedure~\eqref{eq::dist_pipage}. Then, 
$[\vect{y}_{ii}(|\mathcal{P}_i|)]_r \in \{0,1\}, \,\, r \in \mathcal{P}_i$ and $\vect{1}.\vect{y}_{ii}(|\mathcal{P}_i|)=\kappa_i$. Moreover, 
$\bar{\vect{y}} = [\vect{y}_{11}(|\mathcal{P}_1|),\vect{y}_{22}(|\mathcal{P}_2|),\cdots,\vect{y}_{NN}(|\mathcal{P}_N|)]$ is a vertex of $\mathcal{M}$.
\end{prop}
\begin{proof}
Given the definition of $\delta_p(\tau)$ and $\delta_q(\tau)$, at each iteration of~\eqref{eq::dist_pipage}, either $[\vect{y}_{i}(\tau+1)]_p\in\{0,1\}$ or  $[\vect{y}_{i}(\tau+1)]_q\in\{0,1\}$. Moreover, $[\vect{y}_{i}(\tau+1)]_r\in[0,1]$ for any $r\in\mathcal{P}_i$. Consequently,  $[\vect{y}_{ii}(|\mathcal{P}_i|)]_r \in \{0,1\}, \,\, r \in \mathcal{P}_i$.
Next, note that since $\vect{y}_i(0)=\vect{x}_i(T)$, $i\in\mathcal{A}$, by virtue of Lemma~\ref{lem::x_update_property},  we have $\vect{1}.\vect{y}_{ii}(0) =\kappa_i$. Therefore, because~\eqref{eq::dist_pipage} is a zero-sum iteration, we have  $\vect{1}.\vect{y}_{ii}(\tau) =\kappa_i$, $i\in\mathcal{A}$ for any $\tau\in\mathbb{Z}_{\geq0}$, which confirms $\vect{1}.\vect{y}_{ii}(|\mathcal{P}_i|)=\kappa_i$ and $\bar{\vect{y}}\in\mathcal{M}$. Lastly, because $[\vect{y}_{ii}(|\mathcal{P}_i|)]_r \in \{0,1\}, \,\, r \in \mathcal{P}_i$ for any $i\in\mathcal{A}$, $\bar{\vect{y}}$ is a vertex of $\mathcal{M}$.
\end{proof}

Our distributed stochastic Pipage rounding procedure  concludes by each agent $i \in \mathcal{A}$ choosing its suboptimal strategy set according to 
\begin{align}\label{eq::R_bar_i}
&\bar{\mathcal{R}}_i=\mathcal{R}_{\bar{\vect{y}}_i},~~\text{where}\\
&~~\bar{\vect{y}}_i=[\vect{0}_{|\mathcal{P}_1|\times 1}^\top,\cdots,\vect{y}_{ii}(|\mathcal{P}_i|),\cdots,\vect{0}_{|\mathcal{P}_N|\times 1}^\top]^\top,\nonumber
\end{align}
with $\vect{y}_{ii}(|\mathcal{P}_i|)$ obtained from~\eqref{eq::dist_pipage} when it is initialized at $\vect{y}_i(0)=\vect{x}_i(T)$.
The following lemma, whose proof is given in Appendix A, determines the expected utility of the agents if the agents choose their local admissible strategies according to our proposed distributed stochastic Pipage rounding procedure. 

\begin{thm}[Distributed Pipage rounding]\label{lm::expectedInEq}
Let each agents $i\in \mathcal{A}$ choose its strategy set $\bar{\mathcal{R}}_i\in\mathcal{P}_i$ according to~\eqref{eq::R_bar_i}. Let $\bar{\mathcal{R}} = \bigcup_{i \in \mathcal{A}}\bar{\mathcal{R}}_i$. Then, 
\begin{align}\label{eq::rounding_gaurnatee}
    F(\bar{\vect{x}}(T))\leq \mathbb{E}[f(\bar{\mathcal{R}})].
\end{align}
\end{thm}

\subsection{A minimal information implementation}\label{sec::implem}

\begin{algorithm}[t]
\caption{Discrete distributed implementation of the continuous greedy algorithm.} 

\label{alg:the_alg_decen_2}
{\small
\begin{algorithmic}[1]
\STATE ${\textbf{Init:}} ~~\mathcal{F}_1 \gets \emptyset,\cdots, \mathcal{F}_N \gets \emptyset,\,\, t \gets 1$,
\WHILE{$t\leq T$}
\FOR{$i\in \mathcal{A}$}
\STATE Draw $K_i$ sample strategy sets, $\{\mathcal{R}_i^k\}_{k=1}^{K_i}$ such that
 
\hspace{7pt}$q \in\mathcal{R}_i^k$  w.p. $\alpha$ for all $(q,\alpha)\in\mathcal{F}_i$.
\FOR{$p\in \mathcal{P}_i$}
\STATE Compute $w^i_p \sim \mathbb{E}[f(\mathcal{R} \cup \{p\})-f(\mathcal{R}\setminus \{p\})]$ using

  \hspace{15pt} the strategy sample sets of step $4$.
\ENDFOR
\STATE $\{p_1^\star,\cdots,p_{\kappa_i}^\star\}={\arg\max} \,\, (\{w_p^i\}_{p \in \mathcal{P}_i},\kappa_i)$
\STATE 
$\mathcal{F}_i^-\gets \mathcal{F}_i \oplus \left\{(p_1^\star,\frac{1}{T}),\cdots,(p_{\kappa_i}^\star,\frac{1}{T})\right\}$
\STATE Broadcast $\mathcal{F}_i^{-}$ to the neighbors $\mathcal{N}_i$.
\STATE 
$\mathcal{F}_i\gets\underset{j \in \mathcal{N}_i \cup \{i\}}{\textup{MAX}}\mathcal{F}_j^-$
\ENDFOR
\STATE $t \gets t+1$.
\ENDWHILE
\FOR{$i\in \mathcal{A}$}
    \STATE $\bar{\mathcal{R}}_i=\{\bar{p}_{1},\cdots,\bar{p}_{\kappa_i}\}\gets \mathtt{DistStochPipage}(\mathcal{F}_i) $
\ENDFOR
\State Return $\bar{\mathcal{R}}_1,\cdots,\bar{\mathcal{R}}_N$
\end{algorithmic}
}
\end{algorithm}

\begin{algorithm}[t]
\caption{$\mathtt{DistStochPipage}(~)$} 

\label{alg:the_alg_decen_3}
{\small
\begin{algorithmic}[1]
\STATE ${\textbf{Input:}}$~ $\mathcal{F}_i$
\STATE ${\textbf{Init:}} ~~ \bar{\mathcal{R}}_i=\emptyset$
\WHILE{$|\bar{\mathcal{R}}_i|<\kappa_i$}
\STATE pick any $(\alpha_p,p), (\alpha_p,p)\in\mathcal{F}_i$ such that $p,q\in\mathcal{P}_i$
\STATE Set: $\begin{cases}
 \delta_p = \textup{min} \{\alpha_p , 1-\alpha_q\}\\
 \delta_q = \textup{min} \{\alpha_q , 1-\alpha_p\}
 \end{cases}$
 \STATE Update $\begin{cases}
 \begin{cases}
 \alpha_p \gets \alpha_p-\delta_p \\
 \alpha_q \gets \alpha_q+\delta_p
 \end{cases}\textup{w.p.  } \frac{\delta_q}{\delta_p+\delta_q}\\
 \textup{or}\\
  \begin{cases}
 \alpha_q \gets \alpha_q-\delta_q\\
 \alpha_p \gets \alpha_p+\delta_q
 \end{cases}\textup{w.p.  } \frac{\delta_p}{\delta_p+\delta_q}
 \end{cases}$
 \STATE \textbf{if} $\alpha_p=1$ then $\bar{\mathcal{R}}_i\gets \bar{\mathcal{R}}_i \cup\{p\}$,  $\mathcal{F}_i\gets \mathcal{F}_i\backslash \{(\alpha_p,p)\}$
 \STATE \textbf{if} $\alpha_q=1$ then $\bar{\mathcal{R}}_i\gets \bar{\mathcal{R}}_i \cup\{q\}$,  $\mathcal{F}_i\gets \mathcal{F}_i\backslash \{(\alpha_q,q)\}$
\ENDWHILE
\STATE Return $\bar{\mathcal{R}}_i$
\end{algorithmic}
}
\end{algorithm}

Our proposed suboptimal solution to solve problem~\eqref{eq::mainProblem} consists of iterative propagation step~\eqref{eq::updateRule1}, which requires drawing $K_i$ samples of $\mathcal{R}_{\vect{x}_i}\subset\mathcal{P}$ to compute~\eqref{eq::emperic_nabla_F}, and update step~\eqref{eq::updateRule2}, which requires local interaction between neighbors. After $T$ steps, once $\vect{x}_i(T)$ is obtained, each agent $i\in\mathcal{A}$ computes its suboptimal solution from~\eqref{eq::R_bar_i} after running Pipage procedure~\eqref{eq::dist_pipage} locally for $|\mathcal{P}_i|$ steps to compute $\vect{y}_{ii}(|\mathcal{P}_i|)$. In what following by relying on the properties of the updated local copies of the probability membership we outline a minimum information exchange implementation of our distributed solution. The resulted implementation is summarized as the distributed multilinear-extension-based iterative greedy algorithm presented as Algorithm~\ref{alg:the_alg_decen_2}.


We define the local information set of each agent $i$ at time step $t$~as 
\begin{align}\label{eq:infoSetDef}
\mathcal{F}_i(t)\! = \!\big\{ (p,\alpha)\! \in \!\mathcal{P}\!\times\! [0,1] \Big| [\vect{x}_i(t)]_p\not=0~\text{and}~\alpha = [\vect{x}_i(t)]_p \big\}.
\end{align}
Since $\vect{x}_i(0)=\vect{0}$, then $\mathcal{F}_i(0) =\emptyset$.
Introduction of the information set $\mathcal{F}_i(t)$ provides the framework through which the agents only store and communicate the necessary information. Furthermore, it enables the agents to carry out their local computations using the available information in $\mathcal{F}_i(t)$.  

At each agent $i\in\mathcal{A}$, define $w_p^i(t) =\left[\widetilde{\nabla F}(\vect{x}_i(t))\right]_p$, $p\in\mathcal{P}_i$. $w_p^i(t)$, $p\in\mathcal{P}_i$, then is computed from~\eqref{eq::emperic_nabla_F}  using $K_i$ samples of $\{R_i^k(t)\}_{k=1}^{K_i}$ such that $q \in R_i^k(t)$ with the probability of $\alpha$ for all couples $(q,\alpha) \in \mathcal{F}_i(t)$.

It follows from submodularity of $f$ that $f(\mathcal{R}_i^k(t) \!\cup \{p\})-f(\mathcal{R}_i^k(t))\!\setminus \{p\})\geq 0$. Thus, $w_p^i(t)\geq0$, $p\in\mathcal{P}_i$. Consequently, one realization of $\widetilde{\vect{v}}_{i}(t)$ of problem~\eqref{eq::updateVect} is $\vect{1}_{\{p_1^\star,\cdots,p_{\kappa_i}^\star\}}$, where
\begin{align}\label{eq::p_star}
    \{p_1^\star,\cdots,p_{\kappa_i}^\star\}={\arg\max} \,\, (\{w_p^i\}_{p \in \mathcal{P}_i},\kappa_i).
\end{align}

Since $\vect{1}_{\{p_1^\star,\cdots,p_{\kappa_i}^\star\}}$ is a realization of $\widetilde{\vect{v}}_i(t)$, the corresponding realization of propagation rule~\eqref{eq::updateRule1} over the information set $\mathcal{F}_i(t)$ is denoted as 
\begin{align}\label{eq::F_prop}
    \mathcal{F}_i^-(t+1)=\mathcal{F}_i(t) \oplus \left\{\Big(\big\{p_1^\star,\cdots,p_{\kappa_i}^\star\big\},\frac{1}{T}\Big)\right\},
\end{align}
where $\mathcal{F}_i^-(t+1)$ is consistent with the realization of  $\vect{x}_{i}^-(t+1)$ through the membership probability vector to information set conversion relation~\eqref{eq:infoSetDef}.

Instead of the agents sharing $\vect{x}_i^-(t), \,\, i\in \mathcal{A}$ with their neighbors, they can share their local information set with their neighboring agents and execute a max operation over their local  and received information sets as 
\begin{align}\label{eq::F_update}
    \mathcal{F}_i(t+1)=\underset{j \in \mathcal{N}_i \cup \{i\}}{\textup{MAX}}\mathcal{F}_j^-(t+1).
\end{align}


Consequently, through  the membership probability vector to information set conversion relation~\eqref{eq:infoSetDef}, $\mathcal{F}_i(t+1)$ is consistent with a realization of  $\vect{x}_{i}(t+1)$ .

Finally, given the definition of $\mathcal{F}_i(t)$ in~\eqref{eq:infoSetDef} and in light of  Proposition~\ref{prop::rounding_prop}, the the stochastic rounding procedure~\eqref{eq::dist_pipage} and \eqref{eq::R_bar_i} can be implemented according to Algorithm~\ref{alg:the_alg_decen_3}.

In light of the discussion above, Algorithm~\ref{alg:the_alg_decen_2} gives our distributed multi-linear extension based suboptimal solution for problem~\eqref{eq::mainProblem}. The following theorem establishes the optimality bound of $f(\bar{\mathcal{R}})$ where  $\bar{\mathcal{R}}=\bigcup_{i\in \mathcal{A}} \{\bar{\mathcal{R}}_i\}$ is generated through the decentralized Algorithm~\ref{alg:the_alg_decen_2}.

\begin{thm}[Convergence guarantee and suboptimality gap of Algorithm~\ref{alg:the_alg_decen_2}]\label{thm::main}
Let $f:2^{\mathcal{P}} \to \realnonnegative$ be normalized, monotone increasing and submodular set function. Let $\mathcal{R}^\star$ to be the optimizer of problem~\eqref{eq::mainProblem}. Following the distributed Algorithm \ref{alg:the_alg_decen_2}, the admissible strategy set $\bar{\mathcal{R}}$ satisfies 
\begin{align*}
   \frac{1}{c}(1\!-\!{\textup{e}}^{-c})(1\!-\!(2 c\kappa d(\mathcal{G})\!+\!\frac{c\kappa}{2}+1)\frac{\kappa}{T})f(\mathcal{R}^\star) \!\leq\! \mathbb{E}[f(\bar{\mathcal{R}})],
\end{align*}
with probability of at least $1-2T\,n\, \textup{e}^{-\frac{1}{8T^2}\underline{K}}, \,\, \underline{K}=\underset{i \in \mathcal{A}}{\textup{min}}\,K_i.$
\end{thm}

\begin{proof}
Given that the information set propagation rules~\eqref{eq::p_star}, \eqref{eq::F_prop}, and  \eqref{eq::F_update} are a realization of the vector space propagation rules~\eqref{eq::updateVect}, \eqref{eq::updateRule1}, and \eqref{eq::updateRule2}, we can conclude that the vector $\vect{y} = [\vect{y}_1^\top,\cdots,\vect{y}_N^\top]^\top$ defined as 
\begin{align*}
\begin{cases}
[\vect{y}]_p=\alpha,  \quad &(p,\alpha) \in \mathcal{F}_i(T), \quad p \in \mathcal{P}_i\\
[\vect{y}]_p=0, &\textup{Otherwise}
\end{cases}
\end{align*}
is a realization of $\bar{\vect{x}}(T)$ and satisfies $ F(\bar{\vect{x}}(T)) = F(\vect{y}).$
Moreover, sampling a single strategy $\bar{p}_i$ according to $\vect{y}_i$ out of $\mathcal{P}_i$ is equivalent to sampling rule~\eqref{eq::dist_pipage}. Noting that $\vect{y}$ is a realization of $\bar{\vect{x}}(T)$, Lemma~\ref{thm::main} and Theorem~\ref{lm::expectedInEq} leads us to concluding the proof. 
\end{proof}


\begin{rem}[Extra communication for improved optimality gap]
{\it
Replacing the update step~\eqref{eq::updateRule2} with $\vect{x}_i(t+1)=\vect{y}_i(d(\mathcal{G}))$ where  $\vect{y}_i(0)=\vect{x}_i^-(t+1)$ and 
\begin{align*}
    \vect{y}_i(m) = \underset{j \in \mathcal{N}_i \cup \{i\}}{\textup{max}} \vect{y}_j(m-1), \quad m\in\{1,\cdots, d(\mathcal{G})\},
\end{align*}
i.e., starting with $\vect{x}_i^-(t+1)$ and recursively repeating the update step~\eqref{eq::updateRule2} using the output of the previous recursion for $d(\mathcal{G})$ times, 
each agent $i\in\mathcal{A}$ arrives at $\vect{x}_i(t+1)=\bar{\vect{x}}(t+1)$ (recall Lemma~\ref{lem::maxVect}). 
Hence, for this revised implementation, following the proof of Theorem~\ref{thm::main}, we observe that~\eqref{eq::mainthm3} is replaced by
$\left|\frac{\partial F}{\partial [\vect{x}]_p}(\bar{\vect{x}}(t)) - \frac{\partial F}{\partial [\vect{x}]_p}({\vect{x}}_i(t))\right|=0$, which consequently, leads to 
\begin{align}\label{eq::imp_opt_bound}
   \frac{1}{c}(1\!-\!{\textup{e}}^{-c})(1\!-\!(\!\frac{c\,\kappa}{2}+1)\frac{\kappa}{T})f(\mathcal{R}^\star)\!\leq\! F(\bar{\vect{x}}_{ii}(T)),
\end{align}
with the probability of $\left(\prod_{i\in\mathcal{A}} (1-2\textup{e}^{-\frac{1}{8T^2}K_j})^{|\mathcal{P}_i|}\right)^T$. This improved optimality gap is achieved by $(d(\mathcal{G})-1)T$ extra communication per agent. The optimality bound~\eqref{eq::imp_opt_bound} is the same bound that is achieved by the centralized algorithm of~\cite{JV:10}. To implement this revision, Algorithm~\ref{alg:the_alg_decen_2}'s step 11 (equivalent to~\eqref{eq::F_prop}) should be replaced by $ \mathcal{F}_i=\mathcal{H}_i(d(\mathcal{G}))$, where $\mathcal{H}_i(0)= \mathcal{F}_i^{-}$, and 
\begin{align}\label{eq::F_update_thm}
    \mathcal{H}_i(m)=\underset{j \in \mathcal{N}_i \cup \{i\}}{\textup{MAX}}\mathcal{H}_j^-(m-1), \quad m\in\{1,\cdots,d(\mathcal{G})\}.
\end{align}
}
\end{rem}

\setlength{\fboxsep}{1pt}

\begin{figure}
\centering
    {\includegraphics[width=.80\textwidth]{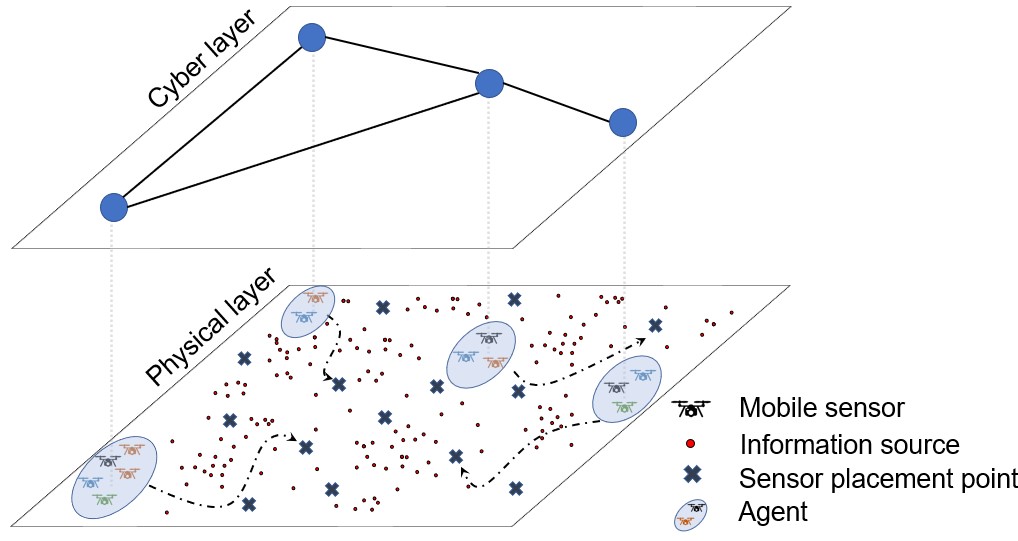}}%
    \caption{{\small   
    A sensor placement problem.}}
    \label{fig:city_partition}%
\end{figure}

\section{Numerical Example}\label{sec::numerical}

\begin{figure}[t]
    \centering
    \includegraphics[width=0.80\textwidth]{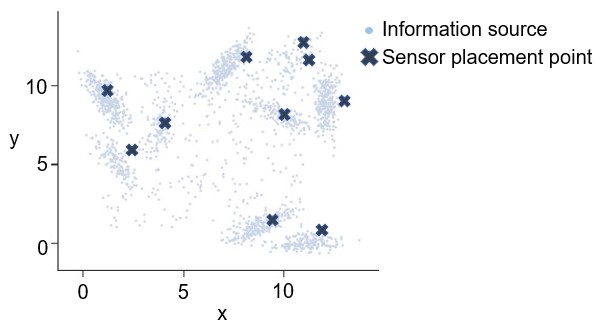}
    \caption{Set of information sources $\mathcal{D}$ and set of sensor placement point $\mathcal{B}$.}
    \label{fig:points}
\end{figure}

\begin{figure}[t]
     \centering
     \begin{subfigure}[t]{0.80\textwidth}
         \centering
         \includegraphics[width=\textwidth]{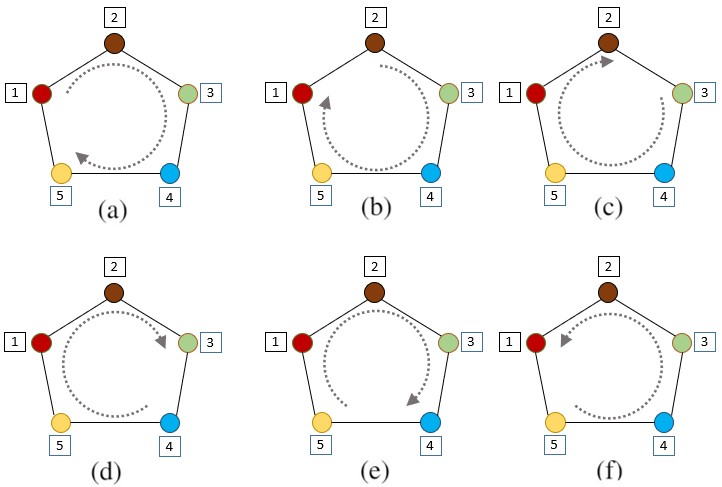}
     \end{subfigure}
        \caption{{\small The communication graph of the agents is a ring graph. Six  possible communication sequences to implement a sequential greedy algorithm is shown.}}
        \label{fig:sequences}
\end{figure}

Consider the sensor placement problem depicted in Fig.~\ref{fig:city_partition}. Each agent $i\in\mathcal{A}$ should deploy $\kappa_i$ mobile sensors at $\mathcal{B}_i \subset \mathcal{B}$ sensor-placement locations on the prespecified information harvesting points $\mathcal{B}$ to harvest information from a set of $\mathcal{D}$ information sources. $\mathcal{B}_1,\cdots,\mathcal{B}_N$ can be overlapping. 
Each information source $d\in\mathcal{D}$ relays its information to the nearest mobile sensor. To maximize the utility of the group the mobile sensors must be placed such that the distance between the information sources and mobile sensors is minimized. This sensor-placement problem can be cast as problem~\eqref{fig:city_partition} where the distinct strategy set of each agent $i\in\mathcal{A}$ is modeled as $\mathcal{P}_i=\{(i,b)|b\in \mathcal{B}_i\}$ and the utility function is 
\begin{align}\label{eq::util_num}
    f(\mathcal{P}) = L(\{\vect{0}\}) - L(\mathcal{P} \cup \{\vect{0}\}),
\end{align}
where $L(\mathcal{P}) =\SUM{d \in \mathcal{D}}{} \underset{(i,b) \in \mathcal{P}}{\textup{min}}  \Lnorm {d-b} \Rnorm$.
Here, with abuse of notation, we use $d$ and $b$, respectively, to indicate also the location of an information source and a mobile sensor. The utility function~\eqref{eq::util_num} is submodular and monotone increasing~\cite{RG-AK:10}. For our numerical study, we consider $2000$ information sources spread in a two-dimensional field with $10$ sensor deployment locations $\mathcal{B}=\{b_1,\cdots,b_{10}\}$, $b_l=(x_l,y_l)$, to deploy mobile sensors as shown in Fig.~\ref{fig:points}. We consider a set of five agents $\mathcal{A}=\{1,2,3,4,5\}$ whose goal is to deploy $\kappa_1=5, \,\,\kappa_2=2, \,\, \kappa_3=1, \,\, \kappa_4=1, \,\, \kappa_5=1$ mobile sensors at  $\mathcal{B}_1=\{b_1\cdots,b_{10}\}$, $\mathcal{B}_2=\{b_1,\cdots,b_5\}$, $\mathcal{B}_3=\{b_1,b_2,b_3\}$, $\mathcal{B}_4=\{b_1,b_2\}$, and $\mathcal{B}_5=\{b_1,b_2\}$. Although, the general form of the problem is NP-hard, we have designed our numerical example such the optimal solution is trivial: agent $1$ places its $\kappa_1=5$ mobile sensors at locations $\{b_6,\cdots,b_{10}\}\subset\mathcal{B}_1$, agent $2$ places its $\kappa_2=2$ mobile sensors at $\{b_4,b_5\}\subset \mathcal{B}_1$, agent $3$ places its $\kappa_3=1$ at $b_3\in\mathcal{B}_3$, agent $2$ places its $\kappa_2=1$ at $b_2\in\mathcal{B}_2$, and agent $1$ places its $\kappa_1=1$ at $b_1\in\mathcal{B}_1$ (agent $1$ and $2$ can also swap their location). This setting allows us to compare the outcome of the suboptimal solutions against the optimal one. Recall that to maximize the utility of the group the mobile sensors must be placed such that the distance between the information sources and mobile sensors is minimized. Thus, the optimal solution is to place all the mobile sensors at distinct sensor-placement locations. 

Let the communication topology of the agents be an undirected ring graph, see Fig.~\ref{fig:sequences}. First, we solve the problem using a decentralized sequential greedy algorithm  following~\cite{NR-SSK:21}. That is, we choose a route $\mathtt{SEQ}$ that visits all the agents and makes the agents perform the sequential greedy algorithm by sequential message passing according to $\mathtt{SEQ}$. Fig.~\ref{fig:sequences}(a)-(f) gives $6$ of the possible $\mathtt{SEQ}$ depicted by the semi-circular arrow inside the networks. Next, we solve the problem using Algorithm~\ref{alg:the_alg_decen_2}, which is a modest number of iteration ($T=50$) and using a modest number of samples ($K_i=1000$) in fact finds almost the optimal solution, see Fig.~\ref{fig:average2}. on the other hand, as Fig.\ref{fig:average} shows, the performance of the sequential greedy algorithm depends on what $\mathtt{SEQ}$ agents follow, with $\mathtt{SEQ}$ of Fig.~\ref{fig:sequences}(a) delivering the worst performance. We can attribute this inconsistency to the heterogeneity of the agents' mobile sensor numbers. When agents with larger number of choices pick first, this limits the options of the agents with a lower number of sensors available. However, the performance of Algorithm~\ref{alg:the_alg_decen_2} is regardless of any particular path on the graph. Through its iterative process, the agents get the chance to readjust their choices. Intuitively, this explains the better optimality gap of the continuous greedy algorithm over the sequential greedy algorithm. 
\begin{figure}[t]
     \centering
         \includegraphics[width=0.70\textwidth]{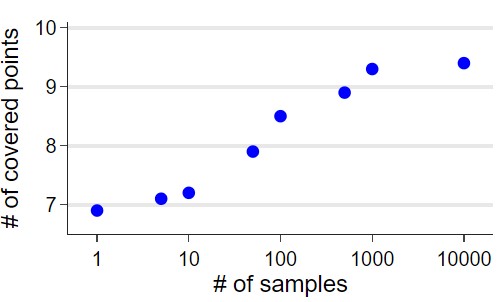}
        \caption{{\small The average number of sensor placement points covered by mobile sensors when Algorithm~\ref{alg:the_alg_decen_2} is implemented by different sample numbers and $T=10$.}}
        \label{fig:average2}
\end{figure}

\begin{figure}[t]
     \centering
         \includegraphics[width=0.70\textwidth]{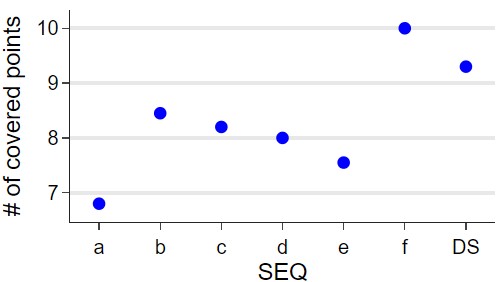}
        \caption{{\small The average number of covered placement points over $50$ difference random generated information source and sensor placement locations. The x-axis corresponds to the six $\mathtt{SEQ}$ in Fig.~\ref{fig:sequences}(a)-(f) and Algorithm~\ref{alg:the_alg_decen_2} denoted by DS. The y-axis corresponds to average number of sensor placement points covered by the mobile agents.}}
        \label{fig:average}
\end{figure}


\section{Conclusion}
We proposed a distributed suboptimal algorithm to solve the problem of maximizing a monotone increasing submodular set function subject to a partition matroid constraint. Our problem of interest was motivated by optimal multi-agent sensor placement problems in discrete space. Our algorithm was a practical decentralization of a multilinear extension-based algorithm that achieves $\frac{1}{c}(1-{\textup{e}}^{-c}-O(1/T))$ optimally gap, which is an improvement over $\frac{1}{1+c}$ optimality gap that the well-known sequential greedy algorithm achieves. In our numerical example, we compared the outcome obtained by our proposed algorithm with a decentralized sequential greedy algorithm that is constructed from assigning a priority sequence to the agents. We showed that the outcome of the sequential greedy algorithm is inconsistent and depends on the sequence. However, our algorithm's outcome, due to its iterative nature intrinsically tended to be consistent, which also explains its better optimally gap over the sequential greedy algorithm. Our future work is to study the robustness of our proposed algorithm to message dropout. 

\section*{Appendix A: Proof of the Results in Section~\ref{sec::Main}}

\begin{proof}[Proof of Lemma \ref{lem::maxVect}]
Since $f$ is monotone increasing and submodular, we have $f(\mathcal{R}_{\vect{x}_i(t)} \cup \{p\})-f(\mathcal{R}_{\vect{x}_i(t)} \setminus \{p\})\geq0$ and hence $\widetilde{\nabla F}(\vect{x}_i(t))$ has positive entries  $\forall i \in \mathcal{A}$. Thus, $\widetilde{\vect{v}}_i(t) \in \mathcal{M}_i$, the optimizer of the optimization \eqref{eq::updateVect} has nonnegative entries. Hence, according to the propagation and update rule \eqref{eq::updateRule1} and \eqref{eq::updateRule2}, we can conclude that $\vect{x}_{ii}(t)$ has increasing elements and only agent $i$ can update it and other agents only copy this value as $\hat{\vect{x}}_{ji}(t)$. Therefor, we can conclude that $[\hat{\vect{x}}_{ji}]_p(t)\leq [\vect{x}_{ii}]_p(t)$ for all $p \in \mathcal{P}_i$ which concludes the~proof. 
\end{proof}

\begin{proof}[Proof of Proposition \ref{thm::convergance}]
$f$ is a monotone increasing and submodular set function therefor $f(\mathcal{R}_{\vect{x}_i(t)} \cup \{p\})-f(\mathcal{R}_{\vect{x}_i(t)} \setminus \{p\})\geq0$ and hence $\widetilde{\nabla F}(\vect{x}_i(t))$ has positive entries  $\forall i \in \mathcal{A}$. Then, because $\widetilde{\vect{v}}_i(t) \in \mathcal{M}_i$, it follows from~\eqref{eq::updateVect} that $\widetilde{\vect{v}}_i(t)$ has non-negative entries, $[\widetilde{\vect{v}}_{i}(t)]_p\geq0$\, which satisfy $\SUM{ p \in \mathcal{P}_i }{}[\widetilde{\vect{v}}_{i}(t)]_p=\kappa_i$. Therefore, it follows from~\eqref{eq::updateRule1} and Lemma~\ref{lem::maxVect} that 
\begin{align}\label{eq::oneStepMax}
    \vect{1}.\vect{x}_{ii}(t+1)=\vect{1}.\vect{x}_{ii}(t)+\frac{\kappa_i}{T},\quad i\in\mathcal{A}.
\end{align}
Using~\eqref{eq::oneStepMax}, we can also write
\begin{align}\label{eq::backwardAdd}
    \vect{1}.\vect{x}_{ii}(t)=\vect{1}.\vect{x}_{ii}(t-d(\mathcal{G}))+\frac{\kappa_i}{T}d(\mathcal{G}),\quad i\in\mathcal{A}.
\end{align}
Furthermore, it follows from Lemma~\eqref{lem::maxVect} that 
for all $\forall p \in \mathcal{P}_i$ and any $j\in\mathcal{A}\backslash\{i\}$, we can write
\begin{align}\label{eq::elementInEq}
   [\vect{x}_{i}(t)]_p \geq[{\vect{x}}_{j}(t)]_p.
\end{align}
Also, since every agent $j\in\mathcal{A}\backslash\{i\}$ can be reached from agent $i\in\mathcal{A}$ at most in $d(\mathcal{G})$ hops, it follows from the propagation and update laws~\eqref{eq::updateRule1} and~\eqref{eq::updateRule2}, for all $\forall p \in \mathcal{P}_i$, for any $j\in\mathcal{A}\backslash\{i\}$ that
\begin{align}\label{eq::elementInEq2}
[{\vect{x}}_{j}(t)]_p \geq [\vect{x}_{i}(t-d(\mathcal{G}))]_p(t-d(\mathcal{G})).
\end{align}
Thus, for $i\in\mathcal{A}$ and $j\in\mathcal{A}\backslash\{i\}$,~\eqref{eq::elementInEq} and~\eqref{eq::elementInEq2} result in
\begin{align}\label{eq::backwardIn}
    \vect{1}.\vect{x}_{ii}(t) \geq \vect{1}.\hat{\vect{x}}_{ji}(t) \geq \vect{1}.\vect{x}_{ii}(t-d(\mathcal{G})).
\end{align}
Next, we can use~\eqref{eq::backwardAdd} and~\eqref{eq::backwardIn} to write
\begin{align}\label{eq::backwardInSep}
    \vect{1}.\vect{x}_{ii}(t) \geq \vect{1}.\hat{\vect{x}}_{ji}(t) \geq \vect{1}.\vect{x}_{ii}(t) - \frac{\kappa_i}{T}d(\mathcal{G}),
\end{align}
for $i\in\mathcal{A}$ and $j\in\mathcal{A}\backslash\{i\}$. Using~\eqref{eq::backwardInSep} for any $i\in\mathcal{A}$ we can write 
\begin{align}\label{eq::backwardDouInEq}
    &\sum\nolimits_{l \in \mathcal{A}} \vect{1}.\vect{x}_{ll}(t) \geq
   \vect{1}.\vect{x}_{ii}(t)+ \sum\nolimits_{j \in \mathcal{A} \backslash \{i\}}  \vect{1}.\hat{\vect{x}}_{ij}(t) \geq \nonumber \\
   &\sum\nolimits_{l \in \mathcal{A}} \vect{1}.\vect{x}_{ll}(t) - \frac{\kappa_l}{T}d(\mathcal{G}).
\end{align}
Then, using Lemma~\ref{lem::maxVect}, from \eqref{eq::backwardDouInEq} we can write
\begin{align*}
    \vect{1}.\bar{\vect{x}}(t) \geq \vect{1} .\vect{x}_i(t) \geq  \vect{1}.\bar{\vect{x}}(t) - \frac{\kappa}{T}d(\mathcal{G}),
\end{align*}
with $\kappa = \SUM{i\in \mathcal{A}}{}\kappa_i$, which ascertains~\eqref{eq::Pro51_a}. Next, note that from Lemma~\ref{lem::maxVect}, we have $\vect{x}_{jj}(t)=\vect{x}_{ii}^{-}(t)$ for any $i\in\mathcal{A}$. Then, using~\eqref{eq::updateRule1} and invoking Lemma~\ref{lem::maxVect}, we  obtain~\eqref{eq::Pro51_b},
which, given~\eqref{eq::oneStepMax}, also ascertains~\eqref{eq::Pro51_c}.
\end{proof}

\begin{proof}[Proof of Lemma~\ref{lem::x_update_property}]
The proof follows from a mathematical induction argument. The base case $\vect{x}_i(0)=\vect{0} \in \mathcal{M}$ and $\bar{\vect{x}}(0)=\vect{0} \in \mathcal{M}$ is trivially true. We take it to be true that at time $t$ and for each agent $i \in \mathcal{A}$ it hold that $\vect{x}_i(t) \in \mathcal{M}$ with 
\begin{align*}
    &\vect{1}.\vect{x}_{ii}(t) = \frac{t}{T} \kappa_i,~~\text{and}~~\vect{1}.\hat{\vect{x}}_{ij}(t) \leq \frac{t}{T} \kappa_j, \quad j \in \mathcal{A}\setminus \{i\}.
\end{align*}
for $t<T$ and $\widetilde{\vect{v}}_i(t) \in \mathcal{M}_i$ satisfying
\begin{align*}
    &\sum\nolimits_{p \in \mathcal{P}_i}\!\! [\widetilde{\vect{v}}_i(t)]_p \!=\! \kappa_i,~~\text{and}~~\sum\nolimits_{p \in \mathcal{P}_j}\!\! [\widetilde{\vect{v}}_i(t)]_p \!=\! 0, ~ j \!\in\! \mathcal{A} \!\setminus\! \{i\}.
\end{align*}
Since $[\widetilde{\nabla F}(\vect{x}_i(t))]_p\geq0\quad p \in \mathcal{P}$, then by propagation rule~\eqref{eq::updateRule1}, we establish that 
\begin{align*}
    &\vect{1}.\vect{x}^{-}_{ii}(t+1) = \frac{t+1}{T} \kappa_i,\\
    &\vect{1}.\hat{\vect{x}}^{-}_{ij}(t+1) \leq \frac{t}{T} \kappa_j, \quad j \in \mathcal{A} \setminus \{i\}.
\end{align*}
A a result of $\mathcal{M}_i,\, i\in \mathcal{A}$ being disjoint convex subspaces of $\mathcal{M}$, the update rule~\eqref{eq::updateRule2} leads to 
\begin{align*}
    &\vect{1}.\vect{x}_{ii}(t+1) = \frac{t+1}{T} \kappa_i,\\
    &\vect{1}.\hat{\vect{x}}_{ij}(t+1) \leq \frac{t+1}{T} \kappa_j, \quad j \in \mathcal{A}\setminus \{i\}.
\end{align*}
Therefore, we conclude that $\vect{x}_i(t+1) \in \mathcal{M}$. Moreover, by the definition of $\bar{\vect{x}}(t)$ in~\eqref{eq::x_bar_value} and $\mathcal{M}_i,\, i\in \mathcal{A}$ being disjoint convex subspaces of $\mathcal{M}$, we deduct that 
\begin{align*}
    \vect{1}.\bar{\vect{x}}(t) = \sum\nolimits_{i \in \mathcal{A}}\vect{1}.{\vect{x}}_{ii}(t+1) = \frac{t+1}{T}\kappa_i \quad i \in \mathcal{A}, 
\end{align*}
for $t<T$ and therefore, $\bar{\vect{x}}(t+1) \in \mathcal{M}$. We conclude the proof of (a) by induction and trivially (b) follows. 
\end{proof}
\begin{proof}[Proof of Theorem \ref{lm::expectedInEq}]
Consider the distributed Pipage rounding~\eqref{eq::dist_pipage}. Let $\tau_i$ be any arbitrary iteration stage of~\eqref{eq::dist_pipage} for agent $i\in\mathcal{A}$. Recall that we partitioned $\vect{y}_i$, $i\in\mathcal{A}$ as $\vect{y}_{i}(\tau_i) = [\hat{\vect{y}}_{i1}(\tau_i),\cdots,{\vect{y}}_{ii}(\tau_i),\cdots,\hat{\vect{y}}_{iN}(\tau_i)]$. Let 
$$\vect{y}(\tau)=[\vect{y}_{11}(\tau_1),\cdots,\vect{y}_{ii}(\tau_i+\tau),\cdots,\vect{y}_{NN}(\tau_N)]$$ for any $\tau_j\in\mathbb{Z}_{\geq0}$, $j\in\mathcal{A}$, and arbitrary $i\in\mathcal{A}$.
Distributed Pipage rounding~\eqref{eq::dist_pipage} results in 
\begin{align*}
    \vect{y}(\tau+1) = \begin{cases}\vect{y}(\tau) + \delta_p(\tau_i)\,\vect{z},&\text{w.p.}~\frac{\delta_q(\tau_i)}{\delta_p(\tau_i)+\delta_q(\tau_i)}\in[0,1],\\
    \vect{y}(\tau) - \delta_q(\tau_i)\,\vect{z},&\text{w.p.}~\frac{\delta_p(\tau_i)}{\delta_p(\tau_i)+\delta_q(\tau_i)}\in[0,1],\end{cases}
\end{align*}
for a $\vect{z} \in \{-1,0,1\}^n$ that satisfies $[\vect{z}]_p=1, \,\, [\vect{z}]_q=-1$ and $[\vect{z}]_r=0,\,\, r\not=p,q$. Next, note that the directional convexity of the multilinear function in Lemma~\ref{lem::prop_F_2} yields 
\begin{align*}
    F(\vect{y}(\tau))  \leq\,& \frac{\delta_q(\tau_i)}{\delta_p(\tau_i)+\delta_q(\tau_i)}F(\vect{y}(\tau)+\delta_p(\tau_i)\vect{z})+\\ &\frac{\delta_p(\tau_i)}{\delta_p(\tau_i)+\delta_q(\tau_i)}F(\vect{y}(\tau)-\delta_q(\tau_i)\vect{z}).
\end{align*}
 Hence, we can write 
\begin{align}\label{eq::recUp1}
    F(\vect{y}(\tau)) \leq \mathbb{E}[F(\vect{y}(\tau+1))\big|\vect{y}(\tau)].
\end{align}
Next, taking expectation with respect to $\vect{y}(\tau)$, we get 
\begin{align}\label{eq::recUp2}
    \mathbb{E}[F(\vect{y}(\tau))] \leq \mathbb{E}[F(\vect{y}(\tau+1))].
\end{align}
Note that because $\vect{y}(0)|_{\{\tau_j\}_{j=1}^N=\{\vect{0}\}^N}=\bar{\vect{x}}(T)$, we have $\mathbb{E}[F(\vect{y}(0)|_{\{\tau_j\}_{j=1}^N=\{\vect{0}\}^N})]= F(\bar{\vect{x}}(T))$. Consequently, since $\vect{y}(\tau)$ is defined for any arbitrary $\{\tau_j\}_{j=1}^N$, we can conclude~that 
\begin{align}\label{eq::rounding_proof_main}
    F(\bar{\vect{x}}(T)) \leq \mathbb{E}[F(\bar{\vect{y}})],
\end{align}
where
$\bar{\vect{y}} = [\vect{y}_{11}(|\mathcal{P}_1|),\vect{y}_{22}(|\mathcal{P}_2|),\cdots,\vect{y}_{NN}(|\mathcal{P}_N|)].$
Proposition~\ref{prop::rounding_prop} states that $\bar{y}$ is a vertex of $\mathcal{M}$, therefore
$F(\bar{\vect{y}}) = f(\mathcal{R}_{\bar{\vect{y}}})$.
On the other hand, it follows from~\eqref{eq::R_bar_i} that
$\mathcal{R}_{\bar{\vect{y}}} = \bigcup_{i\in \mathcal{A}} \bar{\mathcal{R}}_i.$ Consequently,~\eqref{eq::rounding_gaurnatee} follows  from~\eqref{eq::rounding_proof_main}.
\end{proof}

\begin{proof}[Proof of Theorem \ref{thm::main}] Knowing that $\left|\frac{\partial^2F}{\partial [\vect{x}]_p \partial [\vect{x}]_q}\right| \leq cf(\mathcal{R}^\star)$ from Lemma \ref{lem::prop_F}, and~\eqref{eq::Pro51_c}, 
it follows from Lemma~\ref{lm::aux3} that
$F(\bar{\vect{x}}(t+1)) - F(\bar{\vect{x}}(t)) \geq \nabla F(\bar{\vect{x}}(t)).(\bar{\vect{x}}(t+1)-\bar{\vect{x}}(t))-\frac{\kappa^2}{2T^2}cf(\mathcal{R}^\star)$,
which, given~\eqref{eq::Pro51_b}, leads to
\begin{align}\label{eq::mainthm2}
  &F(\bar{\vect{x}}(t+1)) - F(\bar{\vect{x}}(t)) \geq \nonumber \\
   &\qquad\frac{1}{T} \sum\nolimits_{i \in \mathcal{A}}\! \widetilde{\vect{v}}_i(t).\nabla F(\bar{\vect{x}}(t)) -\frac{\kappa^2}{2T^2}cf(\mathcal{R}^\star).
\end{align}
By definition, $\bar{\vect{x}}(t)\geq \vect{x}_i(t)$ for any  $\forall i \in \mathcal{A}$. Therefore, given~\eqref{eq::Pro51_a}, by invoking 
Lemma~\ref{lm::aux3}, for any  $i \in \mathcal{A}$ we can write
\begin{align}\label{eq::mainthm3}
    &\left|\frac{\partial F}{\partial [\vect{x}]_p}(\bar{\vect{x}}(t)) - \frac{\partial F}{\partial [\vect{x}]_p}({\vect{x}}_i(t))\right|\leq  \frac{\kappa}{T}d(\mathcal{G})cf(\mathcal{R}^\star),
\end{align}
for $p\in \{1,\cdots,n\}$. Recall that at each time step $t$, the realization of $\widetilde{\vect{v}}_i(t)$ in~\eqref{eq::updateVect} that Algorithm~\ref{alg:the_alg_decen_2} uses for $\{p_1^\star,\cdots,p_{\kappa_i}^\star\} \in \mathcal{P}_i$ is
\begin{align}\label{eq::t_v_i}
\widetilde{\vect{v}}_i(t)=\vect{1}_{\{p_1^\star,\cdots,p_{\kappa_i}^\star\}}, \end{align}
for every $i\in\mathcal{A}$. Thus, $\vect{1}.\widetilde{\vect{v}}_i(t)=\kappa_i$, $i\in\mathcal{A}$. Consequently, using~\eqref{eq::mainthm3} we can write
\begin{align}\label{eq::inEquality1}
    &\sum\nolimits_{i \in \mathcal{A}} \widetilde{\vect{v}}_i(t). \nabla F(\bar{\vect{x}}(t)) \geq \nonumber \\
    &\qquad\sum\nolimits_{i \in \mathcal{A}} \widetilde{\vect{v}}_i(t). \nabla F(\vect{x}_i(t)) - \frac{\kappa^2}{T}d(\mathcal{G}))cf(\mathcal{R}^\star).
\end{align}
Next, we let 
$\bar{\vect{v}}_i(t) = \underset{\vect{v} \in \mathcal{M}_i}{\textup{argmax}}\, \vect{v}.\nabla F(\bar{\vect{x}}(t))$ 
and 
$\hat{\vect{v}}_i(t) = \underset{\vect{v} \in \mathcal{M}_i}{\textup{argmax}}\, \vect{v}.\nabla F(\vect{x}_i(t)).$ 
Then, using $\hat{\vect{v}}_i(t). \nabla F(\vect{x}_i(t)) \geq \bar{\vect{v}}_i(t). \nabla F(\vect{x}_i(t))$ 
and $\hat{\vect{v}}_i(t). \nabla F(\vect{x}_i(t)) \geq \tilde{\vect{v}}_i(t). \nabla F(\vect{x}_i(t))$,
$i\in\mathcal{A}$, and ~\eqref{eq::mainthm3} we can also write 
\begin{subequations}\label{eq::determin_inEquality}
\begin{align}
    &\sum\nolimits_{i \in \mathcal{A}} \hat{\vect{v}}_i(t). \nabla F(\vect{x}_i(t)) \geq \SUM{i \in \mathcal{A}}{} \bar{\vect{v}}_i(t). \nabla F(\vect{x}_i(t))\geq \nonumber \\
    &\quad\sum\nolimits_{i \in \mathcal{A}} \bar{\vect{v}}_i(t). \nabla F(\bar{\vect{x}}(t)) - \frac{\kappa^2}{T}d(\mathcal{G})cf(\mathcal{R}^\star),\label{eq::inEquality4}\\
    &\sum\nolimits_{i \in \mathcal{A}}\!\!\! \hat{\vect{v}}_i(t). \nabla F(\vect{x}_i(t)) \!\geq\! \sum\nolimits_{i \in \mathcal{A}}\!\!\! \tilde{\vect{v}}_i(t). \nabla F(\vect{x}_i(t)).\label{eq::determin_inEquality_b}
\end{align}
\end{subequations}

On the other hand, by virtue of Lemma~\ref{lem::gradient_estimate}, $\left[\widetilde{\nabla F}(\vect{x}_i(t))\right]_p$, $p \in \mathcal{P}_i$ that each agent $i\in\mathcal{A}$ uses to solve optimization problem~\eqref{eq::p_star} (equivalently \eqref{eq::updateVect}) satisfies \begin{align}\label{eq::F_tilde_F}
    \left|\left[\widetilde{\nabla F}(\vect{x}_i(t))\right]_p - \frac{{\partial F}}{\partial [\vect{x}]_p}(\vect{x}_i(t))\right| \leq \frac{1}{2T}f(\mathcal{R}^\star)
\end{align} with the probability of  $1\!-\!2\textup{e}^{-\frac{1}{8T^2}K_i}$. Using~\eqref{eq::determin_inEquality_b} and~\eqref{eq::F_tilde_F}, and because the samples are drawn independently, we obtain 
\begin{subequations}
\begin{align}\label{eq::inEquality2}
    &\SUM{i \in \mathcal{A}}{}\! \widetilde{\vect{v}}_i(t). \nabla F(\vect{x}_i(t)) \geq \!\!
    \SUM{i \in \mathcal{A}}{}\! \widetilde{\vect{v}}_i(t). \widetilde{\nabla F}(\vect{x}_i(t))\! -\! \frac{\kappa}{2T}f(\mathcal{R}^\star),\\
    &\sum\nolimits_{i \in \mathcal{A}} \widetilde{\vect{v}}_i(t). \widetilde{\nabla F}(\vect{x}_i(t)) \geq
    \sum\nolimits_{i \in \mathcal{A}} \hat{\vect{v}}_i(t). \widetilde{\nabla F}(\vect{x}_i(t))\geq \nonumber \\
    &\quad \sum\nolimits_{i \in \mathcal{A}} \hat{\vect{v}}_i(t). \nabla F(\vect{x}_i(t)) - \frac{\kappa}{2T}f(\mathcal{R}^\star),\label{eq::inEquality3}
\end{align}    
\end{subequations}
with the probability of $\prod_{i\in\mathcal{A}} (1-2\textup{e}^{-\frac{1}{8T^2}K_i})^{|\mathcal{P}_i|}$.

From~\eqref{eq::inEquality1},~\eqref{eq::inEquality4},\eqref{eq::inEquality2}, and \eqref{eq::inEquality3} now we can write 
\begin{align}\label{eq::inEquality5}
    &\sum\nolimits_{i \in \mathcal{A}} \widetilde{\vect{v}}_i(t). \nabla F(\bar{\vect{x}}(t)) \geq \nonumber \\
    &\quad\sum\nolimits_{i \in \mathcal{A}} \!\!\bar{\vect{v}}_i(t). \nabla F(\bar{\vect{x}}(t))\! -\! (2 \kappa d(\mathcal{G}))+1 )\frac{\kappa}{T}f(\mathcal{R}^\star),
\end{align}
with the probability of $1-2\,\sum_{i \in \mathcal{A}}^{}\textup{e}^{-\frac{1}{8T^2}K_i}$.

Next, let $\vect{v}_i^\star$ be the projection of $\vect{1}_{\mathcal{R}^\star}$ into $\mathcal{M}_i$. Knowing that $\mathcal{M}_i$'s are disjoint sub-spaces of $\mathcal{M}$ covering the whole space then we can write 
\begin{align}\label{eq::subspace}
    \vect{1}_{\mathcal{R}^\star} = \sum\nolimits_{i \in \mathcal{A}} \vect{v}_i^\star.
\end{align}
Then, using \eqref{eq::inEquality5}, \eqref{eq::subspace}, and invoking Lemma \ref{lm::optimalInEq} and the fact that $\bar{\vect{v}}_i(t). \nabla F(\bar{\vect{x}}(t)) \geq \vect{v}^\star_i(t). \nabla F(\bar{\vect{x}}(t))$ we obtain
\begin{align}\label{eq::inEquality6}
    &\sum\nolimits_{i \in \mathcal{A}} \widetilde{\vect{v}}_i(t). \nabla F(\bar{\vect{x}}(t)) \geq \nonumber \\ 
    &\quad\sum\nolimits_{i \in \mathcal{A}} \vect{v}^\star_i(t). \nabla F(\bar{\vect{x}}(t)) - (2c \kappa d(\mathcal{G})+1 )\frac{\kappa }{T}f(\mathcal{R}^\star) =\nonumber \\    
    &\quad\vect{1}_{\mathcal{R}^\star}.\nabla F(\bar{\vect{x}}(t)) - (2c\kappa d(\mathcal{G})+1 )\frac{\kappa}{T}f(\mathcal{R}^\star) \geq \nonumber \\ 
    &\quad f(\mathcal{R}^\star) - cF(\bar{\vect{x}}(t))- (2c \kappa d(\mathcal{G})+1)\frac{\kappa}{T}f(\mathcal{R}^\star),
\end{align}
with the probability of $\prod_{i\in\mathcal{A}} (1-2\textup{e}^{-\frac{1}{8T^2}K_i})^{|\mathcal{P}_i|}$. Hence, using~\eqref{eq::mainthm2} and \eqref{eq::inEquality6}, we conclude that
\begin{align}\label{eq::mainthm4}
   &F(\bar{\vect{x}}(t+1)) - F(\bar{\vect{x}}(t)) \geq \frac{1}{T}(f(\mathcal{R}^\star) - cF(\bar{\vect{x}}(t))-\nonumber \\
   &\qquad
   (2c \kappa d(\mathcal{G}))+\frac{c\kappa}{2}+1)\frac{\kappa}{T^2}f(\mathcal{R}^\star),
\end{align} 
with the probability of $\prod_{i\in\mathcal{A}} (1-2\textup{e}^{-\frac{1}{8T^2}K_i})^{|\mathcal{P}_i|}$.
Next, let $g(t)=f(\mathcal{R}^\star)-cF(\bar{\vect{x}}(t))$ and $\beta =(2c \kappa d(\mathcal{G}))+\frac{c\kappa}{2}+1)\frac{\kappa}{T^2}f(\mathcal{R}^\star)$, to rewrite~\eqref{eq::mainthm4} as 
\begin{align}\label{eq::mainthm5}
   &(f(\mathcal{R}^\star)-cF(\bar{\vect{x}}(t))) - (f(\mathcal{R}^\star)-cF(\bar{\vect{x}}(t+1)))=\nonumber \\
   &\quad g(t)-g(t+1) \geq
   \frac{c}{T}(f(\mathcal{R}^\star)-cF(\bar{\vect{x}}(t))) - c\beta  = \nonumber \\
   &\quad\frac{c}{T}g(t)-c\beta.
\end{align}\label{eq::mainthm6}
Then from inequality \eqref{eq::mainthm5} we get 
\begin{align}\label{eq::mainthm61}
   g(t+1) \leq (1-\frac{c}{T})g(t)+c\beta
\end{align}
with the probability of $\prod_{i\in\mathcal{A}} (1-2\textup{e}^{-\frac{1}{8T^2}K_i})^{|\mathcal{P}_i|}$. Solving for inequality \eqref{eq::mainthm61} at time $T$ yields
\begin{align}\label{eq::mainthm7}
    g(T) \leq& \,(1-\frac{c}{T})^T g(0) +\beta \sum\nolimits_{k=0}^{T-1} (1-\frac{c}{T})^k \nonumber \\ 
    =\,&(1-\frac{c}{T})^T g(0) + T\beta(1-(1-\frac{c}{T})^T)
\end{align}
with the probability of 
$\left(\prod_{i\in\mathcal{A}} (1-2\textup{e}^{-\frac{1}{8T^2}K_i})^{|\mathcal{P}_i|}\right)^T$.
Substituting back $g(T)=f(\mathcal{R}^\star)-cF(\bar{\vect{x}}(T))$ and $g(0)=f(\mathcal{R}^\star)-cF(\vect{x}(0))=f(\mathcal{R}^\star)$, in~\eqref{eq::mainthm7} we then obtain 
\begin{align}\label{eq::mainthm8}
    &\frac{1}{c}(1-(1-\frac{1}{T})^T)(f(\mathcal{R}^\star)-T\beta)= \nonumber \\ 
    &\quad\frac{1}{c}(1-(1-\frac{1}{T})^T)(1-(2c \kappa d(\mathcal{G})+\frac{c\kappa}{2}+1)\frac{\kappa}{T})f(\mathcal{R}^\star)\nonumber \\
    &\quad \leq F(\bar{\vect{x}}(T)),
\end{align}
with the probability of $\left(\prod_{i\in\mathcal{A}} (1-2\textup{e}^{-\frac{1}{8T^2}K_i})^{|\mathcal{P}_i|}\right)^T$.
By applying $\textup{e}^{-c}\geq(1-(1-\frac{c}{T})^T)$, we get 
\begin{align}
\frac{1}{c}(1\!-\!{\textup{e}}^{-c})(1\!-\!(2 c\kappa d(\mathcal{G})\!+\!\frac{c\kappa}{2}+1)\frac{\kappa}{T})f(\mathcal{R}^\star) \!\leq F(\bar{\vect{x}}(T)),
\end{align}
with the probability of $\left(\prod_{i\in\mathcal{A}} (1-2\textup{e}^{-\frac{1}{8T^2}K_i})^{|\mathcal{P}_i|}\right)^T$.
which concludes the proof.
\end{proof}

\section*{Appendix B: Auxiliary Results}
\begin{proof}[Proof of Lemma~\ref{lem::gradient_estimate}]
Define the random variable 
\begin{align*}
  \!  X\!\! =\!\! \left(\!\!(f(\mathcal{R}_{\vect{x}}\!\cup \{p\})\!-\!f(\mathcal{R}_{\vect{x}}\!\!\setminus\! \{p\}))\!-\!\frac{\partial F}{\partial x_{p}}(\vect{x})\!\!\right)\!\!\Big/\!\!f(\mathcal{R}^\star),
\end{align*}
and assume that agent $j\in\mathcal{A}$ takes $K_j$ samples from  $\mathcal{R}_{\vect{x}}$ to construct $\{X_k\}_{k=1}^{K}$ realization of $X$. Since $f$ is a submodular function, then we have $(f(\mathcal{R}_{\vect{x}}\!\cup \{p\})\!-\!f(\mathcal{R}_{\vect{x}}\!\!\setminus\! \{p\}))\leq f(\mathcal{R}^\star)$. Consequently using equation~\eqref{eq::firstDer}, we conclude that $0\leq X \leq 1$. Hence, using Theorem~\ref{thm::sampling}, we~have $\left|\sum\nolimits_{k=1}^{K}X_k\right|\geq\frac{1}{2T}K$ with the probability of $2\textup{e}^{-\frac{1}{8T^2}K}$.  Hence, the estimation accuracy of $\nabla F(\vect{x})$, is given by
$
    \left|\left[\widetilde{\nabla F}(\vect{x}_i(t))\right]_p  - \frac{{\partial F}}{\partial [\vect{x}]_p}(\vect{x})\right| \geq \frac{1}{2T}f(\mathcal{R}^\star), ~ p\in \{1,\cdots,n\}$
with the probability of $2\textup{e}^{-\frac{1}{8T^2}K}$.
\end{proof}

\begin{lem}[First and second derivatives of the multilinear extension]\label{lem::prop_F} 
Let $f:2^\mathcal{P}\to\real_{\geq0}$, $\mathcal{P}=\{1,\cdots,n\}$, be increasing and submodular set function with curvature $c$, and the multinear extension function $F(\vect{x})$ defined in~\eqref{eq::F_determin}. Then, $\frac{\partial F}{\partial [\vect{x}]_p}\geq (1-c)f(p)$ for all $p\in\mathcal{P}$ and  $\vect{x}\in[0,1]^n$. Moreover,  $-cf(\mathcal{R}^\star)\leq \frac{\partial^2 F}{\partial [\vect{x}]_p [\vect{x}]_q}\leq 0$ for all $p,q\in\mathcal{P}$ and  $\vect{x}\in[0,1]^n$.
\end{lem}
\begin{proof}
The derivative of $F(x)$ can be written as 
\begin{align}\label{eq::firsec1}
    \frac{\partial F}{\partial [\vect{x}]_p} (\vect{x})=\Delta_f(p\,|\mathcal{R}_{\vect{x}} \setminus \{p\}).
\end{align}
Furthermore, by the definition of the total curvature \eqref{eq::totalCurvature} we can write $ c \geq 1-\frac{\Delta_f(p|\mathcal{R}_{\vect{x}} \setminus p)}{f(p)}$,
and by conjunction with equation~\eqref{eq::firsec1}, we have $\frac{\partial F}{\partial [\vect{x}]_p}\geq (1-c)f(p)$
which proves the first part of Lemma. Since $p \not \in \mathcal{R}_{\vect{x}}  \cup \{q\} \!\setminus \!\{p\}$, therefor by the definition of the total curvature \eqref{eq::totalCurvature} we can write
\begin{align}\label{eq::firsec2}
    (1-c)f(\{p\})\!\leq\! \Delta_f(p|\mathcal{R}_{\vect{x}}  \cup \{q\} \setminus \{p_i\})\!\leq\! f(\{p\}).
\end{align}
Moreover, Since $p \not \in \mathcal{R}_{\vect{x}}  \setminus \{p,q\}$, therefor by the definition of the total curvature \eqref{eq::totalCurvature} we can write
\begin{align}\label{eq::firsec3}
    (1-c)f(\{p\})\leq \Delta_f(p|\mathcal{R}_{\vect{x}} \setminus \{p,q\})\leq f(\{p\}).
\end{align}
Knowing that
$\Delta_f(p|\mathcal{R}_{\vect{x}}  \cup \{q\} \setminus \{p\})\!\!=\!\!f(\mathcal{R}_{\vect{x}} \cup \{p,q\})\!\! -\!\!f(\mathcal{R}_{\vect{x}}  \cup \{q\} \setminus \{p\})$
and $\Delta_f(p|\mathcal{R}_{\vect{x}} \setminus \{p,q\})=f(\mathcal{R}_{\vect{x}} \cup \{p\}) \setminus \{q\}) -f(\mathcal{R}_{\vect{x}} \setminus \{p,q\})$, 
the definition of second order derivative of $F$ \eqref{eq::secondDer}, we can be written as 
\begin{align}\label{eq::firsec4}
\frac{\partial F}{\partial [\vect{x}]_p \partial [\vect{x}]_q} \!=\! \mathbb{E}[\Delta_f(p|\mathcal{R}_{\vect{x}}  \cup \{q\} \!\setminus\! \{p\}) - \Delta_f(p|\mathcal{R}_{\vect{x}}\! \setminus \!\{p,q\})].
\end{align}
Putting \eqref{eq::firsec2} and \eqref{eq::firsec3} and \eqref{eq::firsec4} together in conjunction with submodular property of $f$ results in $-cf(\{p\}) \leq \frac{\partial^2F}{\partial [\vect{x}]_q \partial [\vect{x}]_q} \leq 0$.
Knowing that $f(\{p\})\leq f(\mathcal{R}^\star)$ results in proving the second part of Lemma.
\end{proof}

\begin{lem}[Directional Convexity]\label{lem::prop_F_2} Let $f:2^\mathcal{P}\to\real_{\geq0}$, $\mathcal{P}=\{1,\cdots,n\}$, be monotone increasing and submodular set function with a multinear extension function $F(\vect{x})$ defined in~\eqref{eq::F_determin}. Then, for any given $\vect{x} \in [0,1]^n$ and $\vect{w} \in \{-1,0,1\}^n$ where $w_p=1, \,\, w_q=-1$ and $w_l=0,\,\, l\in\{1,\cdots,n\}\backslash\{p,q\}$ for some $p,q\in\{1,\cdots,n\}$,  $F(\vect{x}+\lambda \vect{w}):\mathbb{R} \to \realnonnegative$ is a convex function of $\lambda$.\end{lem}
\begin{proof}
Defining the vector $\vect{w} \in \mathbb{R}^n$ and $w_p=1, \,\, w_q=-1$ and $w_l=0,\,\, l\not=p,q$, then the multilinear extension of set function $f$ in the direction of $\vect{w}$ is defined as
\begin{align*}
    F&(\vect{x}+\lambda \vect{w}) = \\
    &\sum_{\mathcal{R} \subset \mathcal{P}\setminus \{p,q\}}{} \!\!\!\!f(\{p\} \cup \mathcal{R}) ([\vect{x}]_p+\lambda)(1-([\vect{x}]_q-\lambda))\mathbb{P}(\mathcal{R})+\\
    &\qquad \qquad f(\{q\} \cup \mathcal{R}) (1-([\vect{x}]_p+\lambda))([\vect{x}]_q-\lambda)\mathbb{P}(\mathcal{R})+\\
    &\qquad \qquad f(\mathcal{R}) (1-([\vect{x}]_p+\lambda))(1-([\vect{x}]_q-\lambda))\mathbb{P}(\mathcal{R})+\\
    &\qquad \qquad f(\mathcal{R}) \!\! - \!\! f(\{p,q\} \cup \mathcal{R}))([\vect{x}]_p+\lambda)([\vect{x}]_q\!\! -\!\! \lambda)\mathbb{P}(\mathcal{R}).
\end{align*}
with $\mathbb{P}(\mathcal{R})=\prod_{r \in \mathcal{R}}^{}\! x_r \prod_{r \not\in \mathcal{R}}^{}\!(1-x_r)$. Taking the second derivative of $F(\vect{x}+\lambda \vect{w})$ with respect to $\lambda$ yields
\begin{align*}
    \frac{\partial^2F(\vect{x}+\lambda \vect{w})}{\partial \lambda^2} \!\! =\!\!\!\! \sum_{\mathcal{R} \subset \mathcal{P}\setminus \{p,q\}}{} \!\!\!&\!\!\!\!\!\!2\mathbb{P}(\mathcal{R}) (f(\{p\} \cup \mathcal{R}) + f(\{q\} \cup \mathcal{R}) \\
    &- f(\mathcal{R}) - f(\{p,q\} \cup \mathcal{R})).
\end{align*}
The submodularity of $f$ asserts that $ \frac{\partial^2F(\vect{x}+\lambda \vect{w})}{\partial \lambda^2}\geq 0$ and consequently, $F(\vect{x}+\lambda \vect{w})$ is a convex function of $\lambda$.
\end{proof}

\begin{lem}[Interval Bound of Twice differentiable function]\label{lm::aux3}
Consider a twice differentiable function $F(\vect{x}):[0,1]^n \to \mathbb{R}$ which satisfies $\left|\frac{\partial^2F}{\partial [\vect{x}]_p \partial [\vect{x}]_q}\right| \leq \alpha$ for any $p,q\in\{1,\cdots,n\}$.  Then for any $\vect{x}_1,\vect{x}_2 \in \mathbb{R}^n$ satisfying $\vect{x}_2 \geq \vect{x}_1$ and $\vect{1}.({\vect{x}}_2 - {\vect{x}}_1) \leq \beta$ we have
\begin{subequations}\label{eq::mainthm1}
\begin{align}
   & \left|\frac{\partial F}{\partial [\vect{x}]_p}({\vect{x}}_1+\epsilon({\vect{x}}_2-{\vect{x}}_1)) -  \frac{\partial F}{\partial [\vect{x}]_p}({\vect{x}}_1)\right| \leq \epsilon \alpha \beta,\\
   & F({\vect{x}}_2) - F({\vect{x}}_1) \geq 
   \nabla F({\vect{x}}_1).({\vect{x}}_2-{\vect{x}}_1)-\frac{1}{2}\alpha \beta^2,
\end{align}     
\end{subequations}
for $\epsilon \in [0 \,\, 1]$.
\end{lem}

\begin{proof}
Let $\vect{h}_p=[\frac{\partial^2F}{\partial [\vect{x}]_p \partial x_1},\cdots,\frac{\partial^2F}{\partial [\vect{x}]_p \partial x_n}]^\top$. Then, we can write 
\begin{align}\label{eq::aux3_1}
    &\left|\frac{\partial F}{\partial [\vect{x}]_p}({\vect{x}}_1+\epsilon({\vect{x}}_2-{\vect{x}}_1)) -  \frac{\partial F}{\partial [\vect{x}]_p}({\vect{x}}_1)\right| \nonumber \\&=\left|\int_{0}^{\epsilon}
    \vect{h}_p({\vect{x}}_1+\tau({\vect{x}}_2-{\vect{x}}_1)).({\vect{x}}_2-{\vect{x}}_1) \textup{d}\tau\right| \nonumber \\
    &\leq \int_{0}^{\epsilon} \alpha \vect{1}.({\vect{x}}_2-{\vect{x}}_1)\textup{d}\tau = \epsilon \alpha \beta,
\end{align}
Furthermore,
$
    F({\vect{x}}_2)  - F({\vect{x}}_1) = \int_{0}^{1}
   \nabla F({\vect{x}}_1+\epsilon({\vect{x}}_2-{\vect{x}}_1)).(\bar{\vect{x}}(t+1)-\bar{\vect{x}}(t))\textup{d}\epsilon  \geq \int_{0}^{1}
   (\nabla F({\vect{x}}_1)-\epsilon \alpha \beta).({\vect{x}}_2-{\vect{x}}_1)\textup{d}\epsilon = 
   \nabla F({\vect{x}}_1).({\vect{x}}_2\!\!-\!\!{\vect{x}}_1)-\frac{1}{2}\alpha \beta^2$,
with the third line follow from equation~\eqref{eq::aux3_1}, which concludes the proof. 
\end{proof}

\end{document}